\pdfoutput=1
\documentclass[english]{jnsao}
\usepackage[utf8]{inputenc}
\usepackage{csquotes}
\usepackage{babel}
\usepackage{amsmath,amsthm}
\usepackage[nameinlink,capitalise]{cleveref}
\usepackage{tikz}

\manuscriptsubmitted{2020-06-18}
\manuscriptaccepted{2020-12-10}
\manuscriptvolume{1}
\manuscriptnumber{6575}
\manuscriptyear{2020}
\manuscriptdoi{10.46298/jnsao-2020-6575}

\newcommand\norm[1]{\left\Vert#1\right\Vert}

\newcommand\N{\mathbb{N}}
\newcommand\R{\mathbb{R}}

\newcommand{\dist}{\operatorname{dist}}
\newcommand{\dom}{\operatorname{dom}}
\newcommand{\gph}{\operatorname{gph}}
\newcommand{\epi}{\operatorname{epi}}
\newcommand{\tto}{\rightrightarrows}
\DeclareMathOperator*{\argmin}{\operatorname{argmin}}

\crefname{figure}{Figure}{Figures}

\allowdisplaybreaks

\numberwithin{equation}{section}

\usetikzlibrary{arrows}

\tikzset{
    punkt/.style={
           rectangle,
           draw=white, very thick,
           text width=10em,
           minimum height=1.5em,
           text centered}
}

\title{Asymptotic stationarity and regularity for nonsmooth optimization problems}

\author{%
	Patrick Mehlitz
	\thanks{
	Institute of Mathematics, Chair of Optimal Control.
	Brandenburgische Technische Universität Cottbus--Senftenberg, Germany.
	\email{mehlitz@b-tu.de}.
	ORCID: 0000-0002-9355-850X.
	}
}
\shortauthor{P.\ Mehlitz}

\shorttitle{Asymptotic regularity in optimization theory}

\acknowledgements{
    The author would like to thank Gerd Wachsmuth for pointing out the simple projection
	argument in the proof of \cref{thm:polyhedrality_and_AMS_regularity}. 
	Furthermore, the author gratefully acknowledges discussions with Mat\'{u}\v{s} Benko 
	about the relationship between AM- and FJM-stationary points which climaxed in some
	results stated in \cref{sec:asymptotic_stationarity}.
}

\begin{document}

\maketitle

\begin{abstract}
    Based on the tools of limiting variational analysis, we derive a
	 sequential necessary optimality condition for nonsmooth mathematical
	 programs which holds without any additional assumptions.
	 In order to ensure that stationary points in this new sense are already
	 Mordukhovich-stationary, the presence of a constraint qualification 
	 which we call AM-regularity is necessary. We investigate the relationship between
	 AM-regularity and other constraint qualifications from nonsmooth
	 optimization like metric (sub-)regularity of the underlying feasibility
	 mapping. Our findings are applied to optimization problems with geometric
	 and, particularly, disjunctive constraints. This way, it is shown that
	 AM-regularity recovers recently introduced cone-continuity-type constraint
	 qualifications, sometimes referred to as AKKT-regularity, 
	 from standard nonlinear and complementarity-constrained optimization.
	 Finally, we discuss some consequences of AM-regularity for the limiting
	 variational calculus.
	 \\[2ex]
	\\[2ex]
\end{abstract}

\section{Introduction}\label{sec:introduction}

Due to their inherent practical relevance in the context of solution algorithms
for optimization problems, sequential necessary optimality conditions and constraint
qualifications became quite popular during the last decade. A suitable theory has
been developed in the context of standard nonlinear programming,
see e.g.\ \cite{AndreaniFazzioSchuverdtSecchin2019,AndreaniHaeserMartinez2011,
AndreaniMartinezRamosSilva2016,AndreaniMartinezRamosSilva2018,AndreaniMartinezSvaiter2010},
complementarity-constrained programming, see \cite{AndreaniHaeserSecchinSilva2019,Ramos2019},
and nonlinear semidefinite programming, see \cite{AndreaniHaeserViana2020}.
Recently, these concepts were generalized to optimization problems in Banach spaces
in \cite{BoergensKanzowMehlitzWachsmuth2019}.
The main idea behind the concept is that even when a local minimizer of a given optimization
problem is not stationary in classical sense (e.g., a Karush--Kuhn--Tucker point in standard
nonlinear programming), it might be \emph{asymptotically} stationary along a sequence of points
converging to the point of interest without any constraint qualification. Now, the question arises which type of
qualification condition is necessary in order to guarantee that an asymptotically stationary
point is already stationary. This indeed leads to the concept of sequential constraint qualifications.
It has been reported in 
\cite{AndreaniFazzioSchuverdtSecchin2019,AndreaniHaeserSecchinSilva2019,AndreaniMartinezRamosSilva2016,
BoergensKanzowMehlitzWachsmuth2019,Ramos2019}
that such sequential constraint qualifications are comparatively weak in comparison with classical
qualification conditions which makes them particularly interesting.

It is a nearby guess that \emph{sequential} stationarity and regularity might be concepts 
which are quite compatible with the popular tools of limiting variational analysis, see e.g.\ 
\cite{Mordukhovich2006,Mordukhovich2018,RockafellarWets1998} and the references therein. 
Indeed, this has been worked out for mathematical problems with complementarity constraints
and the associated concept of Mordukhovich-stationarity (M-stationarity for short) in the
recent paper \cite{Ramos2019}. However, the ideas obviously will work for other classes of
disjunctive programs like mathematical programs with vanishing, switching, or cardinality
constraints as well. 
It is the purpose of this paper to show that the underlying concepts can be further generalized
to a quite abstract class of optimization problems which covers not only all the aforementioned 
settings but also conic as well as cone-complementarity-constrained 
optimization problems and other mathematical programs with equilibrium constraints
which model amongst others that the feasible points need to solve underlying (quasi-) variational
inequalities. Thus, the theory is likely to possess some extensions to bilevel programming
as well.

In this paper, let us consider the mathematical program
\begin{equation}\label{eq:basic_problem}\tag{P}
	\begin{split}
		f(x)&\,\to\,\min\\
		0&\,\in\,\Phi(x)
	\end{split}
\end{equation}
where $f\colon \R^n\to\R$ is a locally Lipschitz continuous function and 
$\Phi\colon \R^n\rightrightarrows \R^m$
is a set-valued mapping whose graph is closed.
Throughout the paper, let $M:=\{x\in\R^n\,|\,0\in\Phi(x)\}$
denote the feasible set of \eqref{eq:basic_problem}. We assume that this set is nonempty.
Let us point out that the theory of this paper stays correct whenever $\R^n$ and $\R^m$
from above are replaced by finite-dimensional Banach spaces $X$ and $Y$. 
Particularly, the results of this manuscript extend to instances of nonlinear semidefinite
programming comprising optimization problems with semidefinite cone complementarity constraints.
Problems of the general form \eqref{eq:basic_problem} have been considered in e.g.\ 
\cite{Gfrerer2013}, \cite[Section~5.2.3]{Mordukhovich2006}, or \cite[Section~3]{YeYe1997}.
In all these contributions, it has been pointed out that whenever $\bar x\in M$ is a local
minimizer of \eqref{eq:basic_problem} such that the mapping $\Phi$ enjoys the
\emph{metric subregularity} property at $(\bar x,0)$, see \cref{sec:set_valued_maps} for 
a definition and additional references to the literature, then $\bar x$ is indeed an 
M-stationary point of this problem. An easy approach to verify the presence of metric
subregularity is given by checking validity of the stronger \emph{metric regularity}
property since the latter can be carried out with the aid of the so-called Mordukhovich criterion
which is stated in terms of the limiting coderivative of $\Phi$, see \cref{sec:variational_analysis}
for details. The latter, however, might be too restrictive which is why several weaker
sufficient conditions for metric subregularity have been worked out 
in particular problem settings during the last years,
see e.g.\ \cite{BaiYeZhang2019,BenkoCervinkaHoheisel2019,GfrererKlatte2016,GfrererYe2017,
HenrionJouraniOutrata2002,HenrionOutrata2001} and the references therein.

As we will see, the sequential approach to necessary optimality conditions and constraint
qualifications for \eqref{eq:basic_problem} leads to a new regularity concept that we call
\emph{asymptotic} Mordukhovich-regularity (AM-regularity for short). The latter is weaker than metric 
regularity of $\Phi$ and not related to the metric subregularity of this map, see 
\cref{ex:metric_subregularity_does_not_imply_AMS_regularity,ex:AMS_regularity_does_not_imply_metric_subregularity}. It, thus, puts some other
light onto the previously known landscape of constraint qualifications which address
\eqref{eq:basic_problem}. Furthermore, we will demonstrate that this new regularity concept
ensures validity of fundamental calculus rules from limiting variational analysis like the
pre-image and the intersection rule, see \cref{thm:representation_of_limiting_normal_cone}
and \cref{sec:variational_calculus}. 
Besides, we show how AM-regularity specifies in exemplary problem settings. 
It will turn out that it covers
several sequential constraint qualifications from the literature. 
Throughout the manuscript, simple examples and counterexamples visualize applicability and
limits of the obtained theory.

The paper is organized as follows: In \cref{sec:preliminaries}, we present the notation exploited
in this manuscript. Furthermore, we review the necessary essentials of set-valued and variational
analysis. \Cref{sec:asymptotic_concepts} is dedicated to the introduction of the asymptotic stationarity
and regularity concepts of our interest. We first derive a sequential necessary optimality
condition of M-stationarity-type in \cref{sec:asymptotic_stationarity} via a simple penalization
argument. Based on that, we introduce the concept of AM-regularity in 
\cref{sec:asymptotic_regularity} and study its theoretical properties as well as its relationship
to other constraint qualifications. In \cref{sec:decoupling_abstract_constraints}, we investigate
the particular situation where the map $\Phi$ can be split in two parts where one is, again,
modeled with the aid of an abstract set-valued mapping while the other one just describes that the variables need to belong to an abstract
constraint set $C\subset\R^n$ which, in practice, can be imagined as a set of simple variational structure.
We show that whenever the set-valued part of $\Phi$ possesses the \emph{Aubin property}
than a weaker constraint qualification than AM-regularity is sufficient for M-stationarity of
local minimizers. We discuss some applications of our results in \cref{sec:applications}.
First, we apply the concept of AM-regularity to the broad class of mathematical problems
with geometric constraints in \cref{sec:geometric_constraints}.
The even more special class of disjunctive optimization problems, where the definition of AM-regularity
can be essentially simplified, is inspected in \cref{sec:disjunctive_programming}.
In this context, a comparison to sequential constraint
qualifications from the literature will be provided. 
Third, we discuss some consequences of AM-regularity for the limiting variational
calculus in \cref{sec:variational_calculus}.
We close the paper with some
concluding remarks in \cref{sec:conclusions}.

\section{Notation and preliminaries}\label{sec:preliminaries}

\subsection{Basic notation}

Throughout the manuscript, we equip $\R^n$ with the Euclidean norm $\norm{\cdot}$.
For some point $x\in\R^n$ and a scalar $\varepsilon>0$, 
$\mathbb B_\varepsilon(x):=\{y\in\R^n\,|\,\norm{y-x}\leq\varepsilon\}$
represents the closed ball around $x$ of radius $\varepsilon$.
For brevity, we exploit $\mathbb B:=\mathbb B_1(0)$.
Let $A\subset\R^n$ be a nonempty set.
We use
\[
	\dist(x,A):=\inf\limits_{y\in A}\norm{y-x}
	\qquad\qquad
	\Pi(x,A):=\argmin\limits_{y\in A}\norm{y-x}
\]
in order to denote the distance of $x$ to $A$ and the associated set of
projections. For brevity, we make use of $A+x=x+A:=\{x+y\in\R^n\,|\,y\in A\}$.
The set
\[
	A^\circ:=\{z\in\R^n\,|\,\forall y\in A\colon\,y^\top z\leq 0\}
\]
is referred to as the polar cone of $A$. It is a nonempty, closed, convex
cone.
The derivative of a differentiable function $F\colon\R^n\to\R^m$ at $x$ will
be represented by $F'(x)\in\R^{m\times n}$ while, in case $m=1$, we
use $\nabla F(x)\in\R^n$ to denote its gradient at $x$.

\subsection{Properties of set-valued mappings}\label{sec:set_valued_maps}

Let $\Upsilon\colon\R^n\tto\R^m$ be a set-valued mapping.
We exploit
\begin{align*}
	\dom\Upsilon&:=\{x\in\R^n\,|\,\Upsilon(x)\neq\varnothing\}\\
	\gph\Upsilon&:=\{(x,y)\in\R^n\times\R^m\,|\,y\in\Upsilon(x)\}\\
	\ker\Upsilon&:=\{x\in\R^n\,|\,0\in\Upsilon(x)\}
\end{align*}
in order to represent the domain, the graph, and the kernel of $\Upsilon$.
Frequently, we will make use of the sequential outer 
Painlev\'{e}--Kuratowski limit of $\Upsilon$
at some point of interest $\bar x\in\dom\Upsilon$ given by
\[
	\limsup\limits_{x\to\bar x}\Upsilon(x)
	:=
	\left\{y\in\R^m\,\middle|\,
		\begin{aligned}
			&\exists\{x_k\}_{k\in\N}\subset\R^n\,\exists \{y_k\}_{k\in\N}\subset\R^m\colon\\
			&\qquad x_k\to\bar x,\,y_k\to y,\,y_k\in\Upsilon(x_k)\,\forall k\in\N
		\end{aligned}
	\right\}.
\]
For some closed set $A\subset\R^n$, we exploit the indicator map
$\Delta_A\colon\R^n\tto\R^m$ given by
\[
	\forall x\in\R^n\colon\quad
	\Delta_A(x):=
		\begin{cases}
			\{0\}	&x\in A \\ \varnothing	&x\notin A
		\end{cases}
\]
where the dimension of the image space will be clear from the context.

In this manuscript, we will often deal with Lipschitzian properties of set-valued mappings.
Recall that $\Upsilon$ is said to be metrically regular at some point 
$(\bar x,\bar y)\in\gph\Upsilon$ whenever there are neighborhoods $U\subset\R^n$ and
$V\subset\R^m$ of $\bar x$ and $\bar y$, respectively, and some constant $\kappa>0$ such that
\[
	\forall x\in U\,\forall y\in V\colon\quad
	\dist(x,\Upsilon^{-1}(y))
	\leq
	\kappa\,\dist(y,\Upsilon(x))
\]
holds. Above, $\Upsilon^{-1}\colon\R^m\tto\R^n$ is the inverse set-valued 
mapping associated with $\Upsilon$
given by $\Upsilon^{-1}(y):=\{x\in\R^n\,|\,y\in\Upsilon(x)\}$ for all $y\in\R^m$.
Fixing $y:=\bar y$ in the definition of metric regularity, we obtain the notion of metric
subregularity, i.e., $\Upsilon$ is said to be metrically subregular at $(\bar x,\bar y)$ if there
are a neighborhood $U\subset\R^n$ of $\bar x$ and a constant $\kappa>0$ such that
\[
	\forall x\in U\colon\quad
	\dist(x,\Upsilon^{-1}(\bar y))
	\leq
	\kappa\,\dist(\bar y,\Upsilon(x))
\]
is valid. The infimum over all such constants $\kappa$ is referred to as 
the modulus of metric subregularity.  
Let us recall that $\Upsilon$ is said to possess the Aubin property at
$(\bar x,\bar y)$ if there are neighborhoods $U\subset\R^n$ and $V\subset\R^m$ of
$\bar x$ and $\bar y$ , respectively, as well as a constant $\kappa>0$ such that the following
estimate is valid:
\begin{equation}\label{eq:Aubin_property}
	\forall x,x'\in U\colon\quad
	\Upsilon(x)\cap V
	\subset
	\Upsilon(x')+\kappa\norm{x-x'}\mathbb B.
\end{equation}
It is well known that $\Upsilon$ possesses the Aubin property at $(\bar x,\bar y)$ if and only
if $\Upsilon^{-1}$ is metrically regular at $(\bar y,\bar x)$.
In the literature, the Aubin property is often referred to as Lipschitz likeness.
Fixing $x':=\bar x$ in the definition of the Aubin property yields the definition of calmness
of $\Upsilon$ at $(\bar x,\bar y)$. The latter is equivalent to metric subregularity of $\Upsilon^{-1}$
at $(\bar y,\bar x)$.
We refer the interested reader to 
\cite{Ioffe2000,KlatteKummer2002,Mordukhovich2006,RockafellarWets1998}
for an overview of the theory and applications of metric regularity and the Aubin property.
Background information about metric subregularity and calmness can be found in
\cite{BenkoCervinkaHoheisel2019,FabianHenrionKrugerOutrata2010,Gfrerer2013,GfrererKlatte2016,
HenrionJouraniOutrata2002,HenrionOutrata2001,IoffeOutrata2008}.
We would like to mention that polyhedral set-valued mappings, i.e., set-valued mappings whose
graph can be represented as the union of finitely many convex polyhedral sets, are calm at each point
of their graphs, see \cite[Proposition~1]{Robinson1981}. Noting that the inverse of a polyhedral
set-valued mapping is also polyhedral, such set-valued mappings are also metrically subregular at
each point of their graphs.

We finalize this paragraph with the following observation: Whenever $\Upsilon$ possesses the Aubin property at $(\bar x,\bar y)\in\gph\Upsilon$, then we find $\kappa>0$ such that for each sequence
$\{x_k\}_{k\in\N}\subset\R^n$ with $x_k\to\bar x$, we have 
$\dist(\bar y,\Upsilon(x_k))\leq\kappa\norm{x_k-\bar x}$ for sufficiently large $k\in\N$ from
\eqref{eq:Aubin_property}. Particularly, there exists a sequence $\{y_k\}_{k\in\N}$ satisfying
$y_k\to\bar y$ and $y_k\in\Upsilon(x_k)$ for sufficiently large $k\in\N$. Thus, $\Upsilon$ is
so-called inner semicontinuous at $(\bar x,\bar y)$.
Let us also mention that $\Upsilon$ is called inner semicompact at $\bar x$ whenever for each
sequence $\{x_k\}_{k\in\N}\subset\R^n$ with $x_k\to\bar x$, 
there is a bounded sequence $\{y_k\}_{k\in\N}\subset\R^m$
such that $y_k\in\Upsilon(x_k)$ holds for all sufficiently large $k\in\N$.

\subsection{Variational analysis}\label{sec:variational_analysis}

The subsequently introduced tools of variational analysis can be found in
the monographs \cite{Mordukhovich2006,Mordukhovich2018} or \cite{RockafellarWets1998}.

For a closed set $A\subset\R^m$ and a point $\bar x\in A$, we exploit
\begin{align*}
	\mathcal T_A(\bar x)
	:=
	\limsup\limits_{t\searrow 0}\frac{A-\bar x}{t}\qquad
	\widehat{\mathcal N}_A(\bar x)
	:=
	\mathcal T_A(\bar x)^\circ\qquad
	\mathcal N_A(\bar x)
	:=
	\limsup\limits_{x\to\bar x,\,x\in A}\widehat{\mathcal N}_A(x)
\end{align*}
in order to denote the tangent (or Bouligand) cone as well as the regular (or Fr\'{e}chet) and the limiting (or Mordukhovich) normal
cone to $A$ at $\bar x$. 
For each $x\notin A$, we stipulate $\mathcal T_A(x):=\varnothing$, 
$\widehat{\mathcal N}_A(x):=\varnothing$,
and $\mathcal N_A(x):=\varnothing$. 
By definition of the limiting normal cone, it is robust
in the sense that we even have
\[
	\limsup\limits_{x\to\bar x}\mathcal N_A(x)
	=
	\mathcal N_A(\bar x),
\]
see \cite[Proposition~6.6]{RockafellarWets1998}.
We recall that whenever $A$ is convex, then the normal cones from above coincide 
with the standard normal cone of convex analysis, i.e.,
\[
	\widehat{\mathcal N}_A(\bar x)
	=
	\mathcal N_A(\bar x)
	=
	\{v\in\R^n\,|\,\forall x\in A\colon\,v^\top(x-\bar x)\leq 0\}
\]
holds true in this situation.

For some extended real-valued, lower semicontinuous function $\varphi\colon\R^n\to\overline{\R}$, 
we denote its epigraph by
$\epi\varphi:=\{(x,\alpha)\in\R^n\times\R\,|\,\alpha\geq\varphi(x)\}$.
Fixing $\bar x\in\R^n$ with $|\varphi(\bar x)|<\infty$, we introduce the 
limiting and singular subdifferential of $\varphi$ at $\bar x$, respectively, as
\begin{align*}
	\partial\varphi(\bar x)
	&:=
	\left\{v\in\R^n\,\middle|\,(v,-1)\in\mathcal N_{\epi\varphi}(\bar x,\varphi(\bar x))\right\}\\
	\partial^\infty\varphi(\bar x)
	&:=
	\left\{v\in\R^n\,\,\middle|\,(v,0)\in\mathcal N_{\epi\varphi}(\bar x,\varphi(\bar x))\right\}.
\end{align*}
It is well known that $\varphi$ is
locally Lipschitz continuous at $\bar x$ if and only if $\partial^\infty\varphi(\bar x)=\{0\}$
holds.

Next, for a set-valued mapping $\Upsilon\colon\R^n\tto\R^m$ with closed graph and some point
$(\bar x,\bar y)\in\gph\Upsilon$, we define the (limiting) coderivative 
$D^*\Upsilon(\bar x,\bar y)\colon\R^m\tto\R^n$ of $\Upsilon$ at $(\bar x,\bar y)$ as stated below:
\[
	\forall y^*\in\R^m\colon\quad
	D^*\Upsilon(\bar x,\bar y)(y^*)
	:=
	\left\{x^*\in\R^n\,\middle|\,(x^*,-y^*)\in\mathcal N_{\gph\Upsilon}(\bar x,\bar y)\right\}.
\]
In case where $\upsilon\colon \R^n\to\R^m$ is a single-valued mapping, we exploit
$D^*\upsilon(\bar x)(y^*):=D^*\upsilon(\bar x,\upsilon(\bar x))(y^*)$ for all $y^*\in\R^m$.
If $\upsilon$ is continuously differentiable at $\bar x$, 
$D^*\upsilon(\bar x)(y^*)=\{\upsilon'(\bar x)^\top y^*\}$ is valid for all $y^*\in\R^m$.

Using the concept of coderivatives, 
it is possible to characterize the presence of metric regularity or the
Aubin property for $\Upsilon$ at $(\bar x,\bar y)\in\gph\Upsilon$. 
More precisely, $\Upsilon$ possesses the Aubin property at $(\bar x,\bar y)$ if and only if
\[
	D^*\Upsilon(\bar x,\bar y)(0)=\{0\}
\]
holds, see \cite[Theorem~4.10]{Mordukhovich2006}. Noting that we have
\[
	\mathcal N_{\gph\Upsilon^{-1}}(\bar y,\bar x)
	=
	\left\{
		(y^*,x^*)\in\R^m\times\R^n\,\middle|\,(x^*,y^*)\in\mathcal N_{\gph\Upsilon}(\bar x,\bar y)
	\right\}
\]
from the change-or-coordinates formula of limiting normals, see \cite[Theorem~1.17]{Mordukhovich2006},
while $\Upsilon$ is metrically regular at $(\bar x,\bar y)$ if and only if $\Upsilon^{-1}$ possesses
the Aubin property at $(\bar y,\bar x)$, the above result also implies that $\Upsilon$ is
metrically regular at $(\bar x,\bar y)$ if and only if the condition
\[
	\ker D^*\Upsilon(\bar x,\bar y)=\{0\}
\]
holds.  This result can be distilled from \cite[Theorem~4.18]{Mordukhovich2006} as well.
Both criteria are referred to as \emph{Mordukhovich criterion} in the literature.

Below, we present a simple calculus rule for the coderivative of set-valued mappings
of certain product structure.
\begin{lemma}\label{lem:product_rule}
	Let $\Gamma\colon \R^n\tto \R^m$ be a set-valued mapping with closed graph.
	Furthermore, let $C\subset\R^n$ be a nonempty, closed set. 
	Let $\Psi\colon\R^n\tto \R^m\times\R^n$ be the set-valued mapping given by
	\[
		\forall x\in\R^n\colon\quad
		\Psi(x):=\begin{pmatrix}\Gamma(x)\\x-C\end{pmatrix}.
	\]
	For a fixed point $(\bar x,(\bar y,\bar z))\in\gph\Psi$, it holds
	\[
		\forall y^*\in\R^m\,\forall z^*\in\R^n\colon\quad
		D^*\Psi(\bar x,(\bar y,\bar z))(y^*,z^*)
		=
		\begin{cases}
			D^*\Gamma(\bar x,\bar y)(y^*)+z^*	&z^*\in\mathcal N_C(\bar x-\bar z)\\
			\varnothing							&\text{otherwise.}
		\end{cases}			
	\]
\end{lemma}
\begin{proof}
	Introducing a linear map $\psi\colon\R^n\times \R^m\times\R^n\to\R^n\times\R^m\times\R^n$ by
	$\psi(x,y,z):=(x,y,x-z)$ for all $x,z\in\R^n$ and $y\in\R^m$, we have
	\[
		\gph\Psi=\{(x,y,z)\,|\,\psi(x,y,z)\in\gph\Gamma\times C\}.
	\]
	Noting that the derivative of $\psi$ is a constant invertible matrix, the desired result follows
	by elementary calculations from the change-of-coordinates formula from
	\cite[Theorem~1.17]{Mordukhovich2006} and the product rule for the computation
	of limiting normals, see \cite[Proposition~1.2]{Mordukhovich2006}.
\end{proof}

\subsection{Generalized distance functions}

In our analysis, we will make use of the distance function to a moving set.
Therefore, let $\Gamma\colon \R^n\tto\R^m$ be a set-valued mapping with closed graph and consider
\[
	\forall x\in\R^n\,\forall y\in\R^m\colon\quad
	\rho_\Gamma(x,y):=\inf\limits_{z\in\Gamma(x)}\norm{y-z}.
\]
The function $\rho_\Gamma\colon\R^n\times\R^m\to\overline{\R}$
has been studied in several different publications,
see e.g.\ \cite{MordukhovichNam2005,Rockafellar1985,Thibault1991} and the
references therein.
In contrast to the classical distance function, see \cite[Section~2.4]{Clarke1983}, $\rho_\Gamma$ is not Lipschitz continuous in general.
In fact, it does not even need to be continuous. 
However, as we will show below, this function is lower semicontinuous since $\Gamma$ possesses
a closed graph.
\begin{lemma}\label{lem:lower_semicontinuity_of_generalized_distance_function}
	Let $\Gamma\colon \R^n\tto\R^m$ be a set-valued mapping with closed graph.
	Then the associated function $\rho_\Gamma$ is lower semicontinuous.
\end{lemma}
\begin{proof}
	Suppose that there exists a point $(\bar x,\bar y)\in\R^n\times \R^m$ where $\Gamma$ is not 
	lower semicontinuous. 
	Then we find sequences $\{x_k\}_{k\in\N}\subset\R^n$ and $\{y_k\}_{k\in\N}\subset\R^m$
	as well as $\alpha\geq 0$ with $x_k\to\bar x$, $y_k\to\bar y$, and 
	$\rho_\Gamma(x_k,y_k)\to\alpha<\rho_\Gamma(\bar x,\bar y)$.
	Particularly, we can assume w.l.o.g.\ that $\Gamma(x_k)\neq\varnothing$ holds for all $k\in\N$.
	Noting that $\Gamma(x_k)$ is closed for each $k\in\N$, we find points $z_k\in\Pi(y_k,\Gamma(x_k))$.
	This yields $\rho_\Gamma(x_k,y_k)=\norm{y_k-z_k}$ for all $k\in\N$. Due to
	\[
		\norm{z_k}\leq\norm{z_k-y_k}+\norm{y_k}=\rho_\Gamma(x_k,y_k)+\norm{y_k}
	\]
	and the boundedness of $\{\rho_\Gamma(x_k,y_k)\}_{k\in\N}$ and $\{y_k\}_{k\in\N}$,
	$\{z_k\}_{k\in\N}$ is bounded as well and possesses an accumulation point $\bar z$.
	Due to $\rho_\Gamma(x_k,y_k)\to\alpha$, we have $\alpha=\norm{\bar y-\bar z}$.
	Observing that $z_k\in\Gamma(x_k)$ holds true for all $k\in\N$, the closedness of $\gph\Gamma$
	yields $\bar z\in\Gamma(\bar x)$. Thus, we have $\rho_\Gamma(\bar x,\bar y)\leq\norm{\bar y-\bar z}=\alpha$
	which is a contradiction.
\end{proof}

Now, we want to identify situations where $\rho_\Gamma$ is a locally Lipschitz continuous function.
Furthermore, we aim for an upper estimate of the limiting subdifferential of this function
which holds at in-set points from $\gph\Gamma$ but also at out-of-set points.
\begin{lemma}\label{lem:generalized_distance_function}
	Let $\Gamma\colon\R^n\tto\R^m$ be a set-valued mapping with closed graph and fix 
	a point $(\bar x,\bar y)\in\R^n\times \R^m$ such that $\bar x\in\dom\Gamma$.
	Then the following assertions hold.
	\begin{enumerate}
		\item[(a)] Assume that $\Gamma$ possesses the Aubin property at all points $(\bar x,y)$ 
			satisfying $y\in\Pi(\bar y,\Gamma(\bar x))$.
			Then $\rho_\Gamma$ is locally Lipschitz continuous at $(\bar x,\bar y)$.
		\item[(b)] The following upper estimate for the limiting subdifferential does always hold:
			\[
				\partial\rho_\Gamma(\bar x,\bar y)\subset\bigcup\limits_{y\in\Pi(\bar y,\Gamma(\bar x))}\mathcal N_{\gph\Gamma}(\bar x,y).
			\]
	\end{enumerate}
\end{lemma}
\begin{proof}
	\begin{enumerate}
		\item[(a)] First, assume that $(\bar x,\bar y)\in\gph\Gamma$ holds. 
			Then we clearly have $\Pi(\bar y,\Gamma(\bar x))=\{\bar y\}$.
			Due to the assumptions of the lemma, $\Gamma$ possesses the Aubin property at $(\bar x,\bar y)$.
			Thus, we can invoke \cite[Theorem~2.3]{Rockafellar1985} in order to obtain the Lipschitz continuity of
			$\rho_\Gamma$ at $(\bar x,\bar y)$.\\
			Next, we assume that $(\bar x,\bar y)\notin\gph\Gamma$ holds. 
			In this case, \cite[Theorem~4.9, Corollary~4.10]{MordukhovichNam2005} guarantee validity of the estimate
			\[
				\partial^\infty\rho_\Gamma(\bar x,\bar y)
				\subset
				\bigcup\limits_{y\in\Pi(\bar y,\Gamma(\bar x))}\{(\xi,0)\,|\,\xi\in D^*\Gamma(\bar x,y)(0)\}.
			\]
			Noting that $\Gamma$ possesses the Aubin property at all points $(\bar x,y)$ with $y\in\Pi(\bar y,\Gamma(\bar x))$,
			the Mordukhovich criterion ensures $D^*\Gamma(\bar x,y)(0)=\{0\}$ which is why we obtain
			$\partial^\infty\rho_\Gamma(\bar x,\bar y)=\{(0,0)\}$ from the above formula.
			Due to \cref{lem:lower_semicontinuity_of_generalized_distance_function}, we already know that $\rho_\Gamma$ is
			lower semicontinuous. Combining these two properties, we obtain that
			$\rho_\Gamma$ is locally Lipschitz continuous at $(\bar x,\bar y)$.
		\item[(b)] If we have $(\bar x,\bar y)\in\gph \Gamma$, then \cite[Proposition~2.7]{Thibault1991} guarantees
			\[
				\mathcal N_{\gph\Gamma}(\bar x,\bar y)=\bigcup_{\alpha\geq 0}\alpha\partial \rho_\Gamma(\bar x,\bar y).
			\]
			On the other hand, in case $(\bar x,\bar y)\notin\gph\Gamma$, \cite[Theorem~4.9, Corollary~4.10]{MordukhovichNam2005}
			can be applied in order to find the estimate
			\[
				\partial\rho_\Gamma(\bar x,\bar y)
				\subset
				\bigcup\limits_{y\in\Pi(\bar y,\Gamma(\bar x))}
				\left\{(\xi,\upsilon)\in\mathcal N_{\gph\Gamma}(\bar x,y)\,\middle|\,\norm{\upsilon}=1\right\}.
			\]
			Taking both formulas together, we obtain the desired general estimate.
	\end{enumerate}
\end{proof}

\section{Asymptotic M-stationarity conditions and asymptotic regularity}\label{sec:asymptotic_concepts}

\subsection{Asymptotic M-stationary conditions}\label{sec:asymptotic_stationarity}

Let $\bar x\in M$ be a local minimizer of \eqref{eq:basic_problem}.
Under suitable assumptions, so-called constraint qualifications, 
one can guarantee that this ensures the existence
of a multiplier $\lambda\in\R^m$ such that
\begin{equation}\label{eq:M_Stationarity}
	0\in\partial f(\bar x)+D^*\Phi(\bar x,0)(\lambda)
\end{equation}
holds, see e.g.\ \cite[Theorem~5.48]{Mordukhovich2006}. 
We will refer to this condition as the Mordukhovich-stationarity
condition (M-stationarity condition for short) of \eqref{eq:basic_problem}.
Now, the question arises whether it is possible to find a milder condition
which holds for each local minimizer of \eqref{eq:basic_problem} even in the
absence of a constraint qualification. A potential candidate for such a
condition could be an \emph{asymptotic} version of M-stationarity which holds
along a sequence of points $\{x_k\}_{k\in\N}$ converging to the local minimizer
of interest. However, one has to specify what \emph{asymptotic} means in 
this regard. The following definition provides a potential and, as we will
see later, reasonable answer to this question.
\begin{definition}\label{def:Asymptotic_M_stationary_point}
	Let $\bar x\in M$ be a feasible point of \eqref{eq:basic_problem}.
	Then we call $\bar x$ an
	\emph{asymptotically Mordukhovich-stationary point} (AM-stationary point) of \eqref{eq:basic_problem} 
	whenever there exist sequences $\{x_k\}_{k\in\N},\{\varepsilon_k\}_{k\in\N}\subset \R^n$ as well as
	$\{y_k\}_{k\in\N},\{\lambda_k\}_{k\in\N}\subset\R^m$ such that
	\begin{equation}\label{eq:AMS_points}
		\forall k\in\N\colon\quad
		\varepsilon_k\in\partial f(x_k)+D^*\Phi(x_k,y_k)(\lambda_k)
	\end{equation}
	as well as $x_k\to\bar x$, $\varepsilon_k\to 0$, and $y_k\to 0$ hold.
	This implicitly requires $\{(x_k,y_k)\}_{k\in\N}\subset\gph\Phi$.
\end{definition}
Observe that in the above definition, no convergence of the multiplier sequence $\{\lambda_k\}_{k\in\N}$
is postulated. Indeed, if it would be bounded, then one could simply take the limit $k\to\infty$
along a subsequence in \eqref{eq:AMS_points}
in order to recover the M-stationarity condition from \eqref{eq:M_Stationarity},
see \cref{lem:relation_to_FJM_stationarity} below.

Using a simple penalization argument, we obtain the following result which shows that each
local minimizer of \eqref{eq:basic_problem} is an AM-stationary point without
any additional assumptions.
\begin{theorem}\label{thm:local_minimizers_AMS_points}
	Let $\bar x\in M$ be a local minimizer of \eqref{eq:basic_problem}.
	Then $\bar x$ is an AM-stationary point of \eqref{eq:basic_problem}.
\end{theorem}
\begin{proof}
	Let $\varepsilon>0$ be chosen such that $f(x)\geq f(\bar x)$ holds for
	all $x\in M\cap\mathbb B_\varepsilon(\bar x)$.
	Consider the 
	optimization problem
	\begin{equation}\label{eq:penalized_basic_problem}\tag{P$(k)$}
	\begin{aligned}
		f(x)+\frac{k}{2}\norm{y}^2+\frac12\norm{x-\bar x}^2&\to\,\min\limits_{x,y}\\
		(x,y)&\,\in\,\gph\Phi\cap(\mathbb B_\varepsilon(\bar x)\times\mathbb B)
	\end{aligned}
	\end{equation}
	which depends on the parameter $k\in\N$. 
	Observe that the objective function of this optimization problem is locally Lipschitz
	continuous while its feasible set is nonempty and compact.
	Consequently, \eqref{eq:penalized_basic_problem} possesses a global minimizer 
	$(x_k,y_k)\in\R^n\times \R^m$
	for each $k\in\N$. 
	Due to $\{(x_k,y_k)\}_{k\in\N}\subset \mathbb B_\varepsilon(\bar x)\times\mathbb B$, 
	this sequence is bounded.
	Choosing a subsequence (if necessary) without relabeling, we can guarantee $x_k\to\tilde x$ for some 
	$\tilde x\in \mathbb B_\varepsilon(\bar x)$ and $y_k\to\tilde y$ for some $\tilde y\in\mathbb B$.	
	Noting that $(\bar x,0)\in\gph\Phi$ is feasible to \eqref{eq:penalized_basic_problem},
	we find
	\begin{equation}\label{eq:penalty_estimate}
		\forall k\in\N\colon\quad 
		f(x_k)+\frac{k}{2}\norm{y_k}^2+\frac12\norm{x_k-\bar x}^2
		\leq
		f(\bar x).
	\end{equation}
	By boundedness of $\{f(x_k)\}_{k\in\N}$, there is a constant $c\in\R$ such that
	$\norm{y_k}^2\leq 2(f(\bar x)-c)/k$ holds for all $k\in\N$.
	Consequently, $\{y_k\}_{k\in\N}$ converges to $0$ as $k\to\infty$, i.e., we have
	$\tilde y=0$.
	The closedness of $\gph\Phi$ now yields $(\tilde x,0)\in\gph\Phi$. 	
	Particularly, we infer $\tilde x\in M\cap\mathbb B_\varepsilon(\bar x)$.
	Now, \eqref{eq:penalty_estimate} and the continuity of all appearing functions yield
	\begin{align*}
		f(\tilde x)+\frac12\norm{\tilde x-\bar x}^2
		&=
		\lim\limits_{k\to\infty}\left(f(x_k)+\frac12\norm{x_k-\bar x}^2\right)\\
		&\leq
		\lim\limits_{k\to\infty}
			\left(
				f(x_k)+\frac{k}{2}\norm{y_k}^2+\frac12\norm{x_k-\bar x}^2
			\right)
		\leq 
		f(\bar x)
		\leq
		f(\tilde x),
	\end{align*}
	and this implies $\tilde x=\bar x$. Particularly, we have $x_k\to\bar x$.
	
	Noting that $(x_k,y_k)$ lies in the interior of $\mathbb B_\varepsilon(\bar x)\times\mathbb B$
	for sufficiently large $k\in\N$, we can apply \cite[Proposition~5.3]{Mordukhovich2006}
	and the subdifferential sum rule from \cite[Theorem~3.36]{Mordukhovich2006} in order to obtain
	\[
		(0,0)\in\partial f(x_k)\times\{0\}+\{(x_k-\bar x,ky_k)\}+\mathcal N_{\gph\Phi}(x_k,y_k)
	\]
	for large enough $k\in\N$. Setting $\lambda_k:=ky_k$ and $\varepsilon_k:=\bar x-x_k$ for
	any such $k\in\N$, we have
	\[
		\varepsilon_k\in\partial f(x_k)+D^*\Phi(x_k,y_k)(\lambda_k),
	\]
	and due to $\varepsilon_k\to 0$, $x_k\to\bar x$, and $y_k\to 0$, $\bar x$ is an
	AM-stationary point of \eqref{eq:basic_problem}.
\end{proof}

Note that in case where the objective function $f$ is differentiable, one could
exploit \cite[Proposition~5.1]{Mordukhovich2006} in the above proof. This way, it would
be possible to replace the limiting coderivative by the regular one (i.e., one replaces the
limiting normal cone to $\gph\Phi$ by the regular normal cone to this set in the definition
of the coderivative) leading to a slightly
stronger concept of asymptotic stationarity. We are, however, interested in taking the
limit $k\to\infty$ in \eqref{eq:AMS_points}, and by definition of the limiting normal
cone and its robustness, it does not matter which of these coderivative constructions is used
in the definition of AM-stationarity since after taking the limit, we obtain a condition
in terms of the limiting coderivative either way.

The above theorem states that in contrast to M-stationarity, AM-stationarity always 
provides a necessary optimality condition for optimization problems of type
\eqref{eq:basic_problem}. The subsequently stated example visualizes this issue.
\begin{example}\label{ex:AMS_point_which_is_no_MSt_point}
	Consider the setting 	
	\begin{align*}
		\forall x\in\R\colon\quad
		f(x):=x\qquad \Phi(x):=[x^2,\infty).
	\end{align*}
	The uniquely determined feasible point $\bar x:=0$ must be the
	global minimizer of the associated program
	\eqref{eq:basic_problem}.
	Exploiting
	\[
		D^*\Phi(x,y)(\lambda)
		=
		\begin{cases}
			\{2\lambda x\}			&y=x^2,\,\lambda\geq 0\\
			\{0\}					&y>x^2,\,\lambda=0\\
			\varnothing				&\text{otherwise}
		\end{cases}
	\]
	for all $x,y,\lambda\in\R$,
	one can easily check that $\bar x$ is not an M-stationary point of this
	program. However, we can set
	\[
		x_k:=-\tfrac1k
		\qquad
		\varepsilon_k:=0
		\qquad
		y_k:=\tfrac1{k^2}
		\qquad
		\lambda_k:=\tfrac{k}{2}
	\]
	for all $k\in\N$
	in order to see that $\bar x$ is an AM-stationary point of the given optimization
	problem. 
	Observe that the multiplier sequence $\{\lambda_k\}_{k\in\N}$ from above is
	not bounded.
\end{example}

It has been shown in \cite[Theorem~8]{Gfrerer2013} that whenever $\bar x\in\R^n$ is
a local minimizer of \eqref{eq:basic_problem}, then, without any additional assumptions,
there exist multipliers $(\lambda_0,\lambda)\in\R\times\R^m$ which satisfy
\begin{equation}\label{eq:FJM_stationarity}
	0\in\lambda_0\partial f(\bar x)+D^*\Phi(\bar x,0)(\lambda)\qquad
	\lambda_0\geq 0\qquad
	\lambda_0+\norm{\lambda}>0.
\end{equation}
Due to the appearance of the leading multiplier $\lambda_0$, one might be tempted to
call $\bar x$ a Fritz--John--Mordukhovich-stationary point (FJM-stationary point)
of \eqref{eq:basic_problem} in this case. 
Observe that whenever we have $\lambda_0>0$ in \eqref{eq:FJM_stationarity}, then $\bar x$
is already M-stationary by positive homogeneity of the coderivative.
\begin{lemma}\label{lem:relation_to_FJM_stationarity}
	Let $\bar x\in\R^n$ be an AM-stationary point of \eqref{eq:basic_problem} such that
	the sequences $\{x_k\}_{k\in\N},\{\varepsilon_k\}_{k\in\N}\subset\R^n$ as well as
	$\{y_k\}_{k\in\N},\{\lambda_k\}_{k\in\N}\subset\R^m$ satisfy \eqref{eq:AMS_points},
	$x_k\to\bar x$, $\varepsilon_k\to 0$, and $y_k\to 0$.
	Then the following assertions hold.
	\begin{enumerate}
		\item[(a)] If $\{\lambda_k\}_{k\in\N}$ is bounded, then $\bar x$ is an 
			M-stationary point of \eqref{eq:basic_problem}.
		\item[(b)] If $\{\lambda_k\}_{k\in\N}$ is not bounded, then we find $\lambda\in\R^m$
			with $\lambda\neq 0$ and $0\in D^*\Phi(\bar x,0)(\lambda)$.
	\end{enumerate}	 
	Particularly, $\bar x$ is an FJM-stationary point of \eqref{eq:basic_problem}.
\end{lemma}
\begin{proof}
	First, we show the statements (a) and (b) separately.
	\begin{enumerate}
	\item[(a)] Due to \eqref{eq:AMS_points}, we find a sequence
	$\{x_k^*\}_{k\in\N}\subset\R^n$ such that 
	\begin{equation}\label{eq:from_AM_to_M_stationarity}
		\forall k\in\N\colon\quad
		\varepsilon_k-x_k^*\in D^*\Phi(x_k,y_k)(\lambda_k)
		\qquad
		x_k^*\in\partial f(x_k)
	\end{equation}
	holds. Noting that the set $\partial f(x)$
	is uniformly bounded in a neighborhood of $\bar x$ by local Lipschitz continuity of
	$f$, see \cite[Corollary~1.81]{Mordukhovich2006}, the sequence $\{x_k^*\}_{k\in\N}$
	is bounded. Let us assume w.l.o.g.\ that there are $x^*\in\R^n$ and $\bar\lambda\in\R^m$
	such that the convergences $x_k^*\to x^*$ and $\lambda_k\to\bar\lambda$ hold. 
	Exploiting the robustness of the limiting normal cone, we find
	$-x^*\in D^*\Phi(\bar x,0)(\bar\lambda)$ and $x^*\in\partial f(\bar x)$ by taking the
	limit $k\to\infty$ in \eqref{eq:from_AM_to_M_stationarity}. Thus, $\bar x$ is an
	M-stationary point of \eqref{eq:basic_problem}.
	\item[(b)] Let us assume w.l.o.g.\ that $\norm{\lambda_k}\to\infty$ holds true. Similar to
	(a), we find a bounded sequence $\{x_k^*\}_{k\in\N}\subset\R^n$ such that 
	\eqref{eq:from_AM_to_M_stationarity} holds. Dividing the first term in
	\eqref{eq:from_AM_to_M_stationarity} by $\norm{\lambda_k}$ and exploiting the positive
	homogeneity of the coderivative, we find
	\begin{equation}\label{eq:from_AM_to_FJM_stationarty}
		\forall k\in\N\colon\quad
		\frac{\varepsilon_k-x_k^*}{\norm{\lambda_k}}
		\in D^*\Phi(x_k,y_k)\left(\frac{\lambda_k}{\norm{\lambda_k}}\right).
	\end{equation}
	Clearly, $\{\lambda_k/\norm{\lambda_k}\}_{k\in\N}$ possesses a non-vanishing accumulation
	point $\lambda\in\R^m$. Thus, we may assume w.l.o.g.\ that $\lambda_k/\norm{\lambda_k}\to\lambda$
	holds. Taking the limit $k\to\infty$ in \eqref{eq:from_AM_to_FJM_stationarty} and
	exploiting the robustness of the limiting normal cone, we find $0\in D^*\Phi(\bar x,0)(\lambda)$.
	Due to $\lambda\neq 0$, the claim follows.
	\end{enumerate}
	Finally, let us comment on FJM-stationarity of the AM-stationary point $\bar x$. 
	In the case (a), $\bar x$ is M- and, thus, FJM-stationary with $\lambda_0:=1$.
	In the setting (b), $\bar x$ is FJM-stationary with $\lambda_0:=0$.
	This completes the proof.
\end{proof}

The example below shows that the concept of AM-stationarity is indeed stronger than FJM-stationarity.
\begin{example}\label{ex:FJM_but_not_AM_stationary}
	We consider the setting
	\[
		\forall x\in\R\colon\quad
		f(x):=-|x|
		\qquad
		\Phi(x):=\{0\}
	\]
	and investigate the point $\bar x:=0$. 
	Clearly, we have $\gph\Phi=\R\times\{0\}$. Furthermore, for arbitrary $x\in\R$ and $\lambda\in\R$,
	we find
	\[
		\partial f(x)\subset\{-1,1\}
		\qquad 
		D^*\Phi(x,0)(\lambda)=\{0\},
	\]
	and this already shows that $\bar x$ is FJM-stationary (with $\lambda_0:=0$ and arbitrary
	$\lambda\neq 0$) but not AM-stationary.
\end{example}

Let $\bar x\in\R^n$ be an FJM-stationary point of \eqref{eq:basic_problem}.
In order to ensure that it is already an M-stationary point, we need to exclude
the case where the leading multiplier equals zero. This corresponds to validity of the
constraint qualification
\[
	0\in D^*\Phi(\bar x,0)(\lambda)
	\quad\Longrightarrow\quad
	\lambda=0
\]
which is equivalent to $\ker D^*\Phi(\bar x,0)=\{0\}$ and, thus, metric regularity
of $\Phi$ at $(\bar x,0)$.
In the subsequent section, we construct a reasonably mild condition, which will be called AM-regularity,
ensuring that a given AM-stationary
point of \eqref{eq:basic_problem} is already M-stationary. Noting that AM-stationarity is
more restrictive than FJM-stationarity, there is some justified hope that AM-regularity
is weaker than metric regularity of $\Phi$, see \cref{thm:metric_regularity_feasibility_map}
where this is actually proven.
Foresightfully, we summarize our results in \cref{fig:stationarities}.
\begin{figure}[h]
\centering
\begin{tikzpicture}[->]

  \node[punkt] at (0,0) 	(A){local minimizer};
  \node[punkt] at (0,-2) 	(B){AM-stationarity};
  \node[punkt] at (0,-4) 	(C){FJM-stationarity};
  \node[punkt] at (6,-2)    (D){M-stationarity};

  \path     (A) edge[-implies,thick,double] node {}(B)
            (B) edge[-implies,thick,double] node {}(C)
            (B) edge[-implies,thick,double, bend left, dashed] node[above] {AM-regularity}(D)
            (D) edge[-implies,thick,double] node {}(B)
            (C) edge[-implies,thick,double, bend right, dashed] node[below right] {metric regularity}(D);
\end{tikzpicture}
\caption{
	Relations between the stationarity conditions addressing \eqref{eq:basic_problem}.
	Dashed relations hold under validity of the mentioned qualification conditions, respectively.
}
\label{fig:stationarities}
\end{figure}
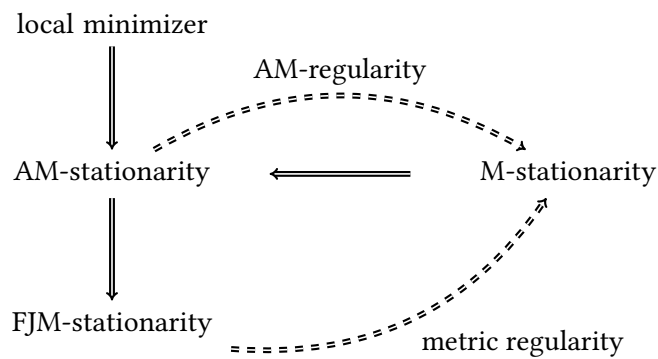

\subsection{Asymptotic regularity}\label{sec:asymptotic_regularity}

We now raise the question under which additional condition a given AM-stationary point
of \eqref{eq:basic_problem} is already an M-stationary point. 
In order to deal with this issue, we make use of the set-valued mapping
$\mathcal M\colon \R^n\times \R^m\tto \R^n$ given by
\[
	\forall x\in\R^n\,\forall y\in\R^m\colon\quad
	\mathcal M(x,y):=\bigcup\limits_{\lambda \in\R^m}D^*\Phi(x,y)(\lambda).
\]
By definition, we have the following result.
\begin{lemma}\label{lem:chararcterizing_AMS_points_via_M}
	Let $\bar x\in M$ be a feasible point of \eqref{eq:basic_problem}.
	Then the following assertions hold.
	\begin{enumerate}
		\item[(a)] If $\bar x$ is an AM-stationary point of \eqref{eq:basic_problem}, then we have
			\begin{equation}\label{eq:consequence_of_AM_stationarity}
				\partial f(\bar x)
				\cap
				\left(-\limsup\limits_{x\to\bar x,\,y\to 0}\mathcal M(x,y)\right)
				\neq
				\varnothing.
			\end{equation}
		\item[(b)] If, on the other hand, $f$ is continuously differentiable at $\bar x$
			while
			\[
				-\nabla f(\bar x)\in\limsup\limits_{x\to\bar x,\,y\to 0}\mathcal M(x,y)
			\]
			holds, then $\bar x$ is an AM-stationary point of \eqref{eq:basic_problem}.
	\end{enumerate}
\end{lemma}
\begin{proof}
	\begin{enumerate}
	\item[(a)] Let $\bar x$ be an AM-stationary point of \eqref{eq:basic_problem}.
	Then we find $\{x_k\}_{k\in\N},\{\varepsilon_k\}_{k\in\N},\{x_k^*\}_{k\in\N}\subset\R^n$,
	and $\{y_k\}_{k\in\N}\subset\R^m$ such that $x_k\to \bar x$, $\varepsilon_k\to 0$, 
	$y_k\to 0$, as well as $x_k^*\in\mathcal M(x_k,y_k)$ and $\varepsilon_k-x_k^*\in\partial f(x_k)$
	for all $k\in\N$ hold. Noting that the set-valued map $x\tto\partial f(x)$ possesses uniformly
	bounded image sets
	around $\bar x$ by local Lipschitz continuity of $f$, 
	see \cite[Corollary~1.81]{Mordukhovich2006},
	the sequence $\{x_k^*\}_{k\in\N}$ needs to be bounded as well and, thus, possesses an
	accumulation point $x^*\in\R^n$ which, by definition, belongs to 
	$\limsup_{x\to\bar x,y\to 0}\mathcal M(x,y)$. 
	On the other hand, $-x^*\in\partial f(\bar x)$ is also true by robustness
	of the limiting normal cone to $\epi f$, i.e., 
	by closedness of the graph associated with the normal cone
	mapping of this set. 
	\item[(b)] From the assumptions, we find $\{x_k\}_{k\in\N}$,
	$\{x_k^*\}_{k\in\N}\subset\R^n$, and $\{y_k\}_{k\in\N}\subset\R^m$ such that $x_k\to \bar x$,
	$y_k\to 0$, and $x_k^*\to-\nabla f(\bar x)$ as well as $x_k^*\in\mathcal M(x_k,y_k)$ for all
	$k\in\N$ hold. Setting $\varepsilon_k:=\nabla f(x_k)+x_k^*$ for each $k\in\N$, we
	have $\varepsilon_k\to 0$ by continuity of 	$\nabla f$ at $\bar x$.
	Thus, the definition of $\mathcal M$ shows that $\bar x$ is an AM-stationary point of
	\eqref{eq:basic_problem}.
	\end{enumerate}
\end{proof}

Observe that statement (b) of the above lemma cannot be generalized to situations where
$f$ is nonsmooth at the point of interest, i.e., condition
\eqref{eq:consequence_of_AM_stationarity} is not necessarily sufficient for a feasible
point $\bar x\in M$ of \eqref{eq:basic_problem} to be AM-stationary.
\begin{example}\label{ex:non_AM_stationary_point}
	Let us consider the setting
	\[
		\forall x\in\R\colon\quad
		f(x):=-|x|\qquad \Phi(x):=[-x^2,\infty)
	\]
	and fix the feasible point $\bar x:=0$ of the associated problem 
	\eqref{eq:basic_problem}. We obtain
	\[
		\partial f(x)
		=
		\begin{cases}
			-1		&x>0\\
			\{-1,1\}	&x=0\\
			1		&x<0
		\end{cases}
		\qquad
		D^*\Phi(x,y)(\lambda)
		=
		\begin{cases}
			\{-2\lambda x\}	&y=-x^2,\,\lambda\geq 0\\
			\{0\}			&y>-x^2,\,\lambda=0\\
			\varnothing		&\text{otherwise}
		\end{cases}
	\]
	for all $x,y,\lambda\in\R$.
	This yields
	\[
		\mathcal M(x,y)
		=
		\begin{cases}
			\R_-	&x>0,\,y=-x^2\\
			\R_+	&x<0,\,y=-x^2\\
			\{0\}	&y>-x^2\;\text{or}\;x=y=0
		\end{cases}
	\]
	for all $x,y\in\R$,	i.e., we find
	\[
		\limsup\limits_{x\to\bar x,\,y\to 0}\mathcal M(x,y)=\R
	\]
	in the present situation, and this shows that \eqref{eq:consequence_of_AM_stationarity} holds.
	On the other hand, we clearly have the inclusion
	$\partial f(x)+\mathcal M(x,y)\subset(-\infty,-1]\cup[1,\infty)$ for all
	$x,y\in\R$, and this clarifies that $\bar x$ cannot be AM-stationary.
\end{example}

By definition of $\mathcal M$, a given feasible point $\bar x\in M$ of
\eqref{eq:basic_problem} is M-stationary if and only if
\[
	\partial f(\bar x)
	\cap
	\left(-\mathcal M(\bar x,0)\right)
	\neq
	\varnothing
\]
holds. Keeping statement (a) of \cref{lem:chararcterizing_AMS_points_via_M} 
in mind, this motivates the subsequent definition.
\begin{definition}\label{def:AMS_regularity}
	A feasible point $\bar x\in M$ of \eqref{eq:basic_problem} is
	said to be \emph{asymptotically Mordukhovich-regular} (AM-regular for short) whenever
	\[
		\limsup\limits_{x\to\bar x,y\to 0}\mathcal M(x,y)
		\subset
		\mathcal M(\bar x,0)
	\]
	is valid.
\end{definition}

By definition, a feasible point $\bar x\in M$ of \eqref{eq:basic_problem} is AM-regular
if and only if the mapping $\mathcal M$ is so-called \emph{outer} (sometimes also
referred to as \emph{upper}) semicontinuous at $(\bar x,0)$ in the sense of set-valued mappings,
see \cite{AubinFrankowska2009,RockafellarWets1998}.

Based on the above observations, the subsequent theorem follows
immediately from \cref{thm:local_minimizers_AMS_points} and
\cref{lem:chararcterizing_AMS_points_via_M}.
It basically says that AM-regularity is a constraint qualification for
\eqref{eq:basic_problem} ensuring M-stationarity of local minimizers.
\begin{theorem}\label{thm:AMS_regularity_CQ}
	Let $\bar x\in M$ be an AM-regular local minimizer of
	\eqref{eq:basic_problem}.
	Then $\bar x$ is an M-stationary point of \eqref{eq:basic_problem}.
\end{theorem}

Next, we want to embed AM-regularity into the landscape of qualification conditions
which address \eqref{eq:basic_problem}. It is well known from 
\cite[Theorem~3]{Gfrerer2013} or \cite[Theorem~5.48]{Mordukhovich2006} 
that metric subregularity of $\Phi$ at $(\bar x,0)$
is enough to guarantee that a local minimizer $\bar x\in M$ of \eqref{eq:basic_problem} 
is an M-stationary point of the latter. 
Using the concept of \emph{directional} metric subregularity, this statement can be
weakened even more, see \cite[Corollary~2]{Gfrerer2013}.
We know that polyhedral set-valued mappings are metrically subregular at all points of their graphs,
i.e., this property already serves as a constraint qualification for \eqref{eq:basic_problem}.
Below, we show that polyhedrality of $\Phi$ is also sufficient for the validity of
AM-regularity.

\begin{theorem}\label{thm:polyhedrality_and_AMS_regularity}
	Let $\Phi$ be a polyhedral set-valued mapping.
	Then each feasible point $\bar x\in M$ of \eqref{eq:basic_problem} is AM-regular.
\end{theorem}
\begin{proof}
	Fix sequences $\{x_k\}_{k\in\N},\{x_k^*\}_{k\in\N}\subset\R^n$ and
	$\{y_k\}_{k\in\N}\subset\R^m$ with $x_k\to\bar x$, $y_k\to 0$, and $x_k^*\to x^*$ for
	some $x^*\in\R^n$ such that $x_k^*\in\mathcal M(x_k,y_k)$ holds for all $k\in\N$.
	For each $k\in\N$, we find $\lambda_k\in\R^m$ such that 
	$(x^*_k,-\lambda_k)\in\mathcal{N}_{\gph\Phi}(x_k,y_k)$ holds.
	Noting that $\gph\Phi$ is the union of finitely many convex polyhedral sets, there only
	exist finitely many regular and, thus, limiting normal cones to the set $\gph\Phi$.
	Particularly, we find a closed cone $\mathcal K\subset\R^n\times\R^m$ such that
	$(x_k^*,-\lambda_k)\in\mathcal K\subset\mathcal N_{\gph\Phi}(x_k,y_k)$ holds along a subsequence (without relabeling). 
	By polyhedrality of $\Phi$, $\mathcal K$ can be represented as the union of
	finitely many convex, polyhedral cones $\mathcal K_1,\ldots,\mathcal K_s\subset\R^n\times\R^m$.
	Again, along a subsequence (without relabeling), we have 
	$(x_k^*,-\lambda_k)\in\mathcal K_i$ for some
	$i\in\{1,\ldots,s\}$.
	Let $P\colon\R^n\times\R^m\to\R^n$ be the projection operator given by $P(x,y):=x$ for all
	$x\in\R^n$ and $y\in\R^m$. Then $P\mathcal K_i$ is polyhedral and, thus, closed by 
	polyhedrality of $\mathcal K_i$,
	i.e., from $\{x_k^*\}_{k\in\N}\subset P\mathcal K_i$ we obtain $x^*\in P\mathcal K_i$.
	This yields the existence of some $\lambda\in\R^m$ such that $(x^*,-\lambda)\in\mathcal K_i$ and,
	thus, $(x^*,-\lambda)\in\mathcal K$. The robustness of the limiting normal cone
	implies $\mathcal K\subset\mathcal N_{\gph \Phi}(\bar x,0)$ due to
	$x_k\to\bar x$ and $y_k\to 0$. Particularly, we have $x^*\in D^*\Phi(\bar x,0)(\lambda)$, i.e.,
	$x^*\in\mathcal M(\bar x,0)$.
	This shows that $\bar x$ is an AM-regular point of \eqref{eq:basic_problem}.
\end{proof}

A natural consequence of the definition of AM-regularity via the
Painlev\'{e}--Kuratowski limit is subsumed in the following lemma.
\begin{lemma}\label{lem:boundedness_of_multipliers}
	Let $\bar x\in M$ be a feasible point of \eqref{eq:basic_problem}.
	Assume that for each sequences $\{x_k\}_{k\in\N},\{x_k^*\}_{k\in\N}\subset\R^n$,
	and $\{y_k\}_{k\in\N}\subset\R^m$ such that $x_k\to\bar x$, $y_k\to 0$, $x_k^*\to x^*$
	for some $x^*\in\R^n$, and
	$x_k^*\in\mathcal M(x_k,y_k)$ for all $k\in\N$ hold, we find a bounded
	sequence $\{\lambda_k\}_{k\in\N}\subset\R^m$ such that $x_k^*\in D^*\Phi(x_k,y_k)(\lambda_k)$
	holds for all $k\in\N$.
	Then $\bar x$ is AM-regular.
\end{lemma}
\begin{proof}
	Let $x^*\in\limsup_{x\to\bar x,y\to 0}\mathcal M(x,y)$ be arbitrarily chosen.
	Then we find $\{x_k\}_{k\in\N},\{x_k^*\}_{k\in\N}\subset\R^n$,
	and $\{y_k\}_{k\in\N}\subset\R^m$ such that $x_k\to\bar x$, $y_k\to 0$, $x_k^*\to x^*$, and
	$x_k^*\in\mathcal M(x_k,y_k)$ for all $k\in\N$ hold.
	By assumption, there is a bounded sequence $\{\lambda_k\}_{k\in\N}\subset\R^m$
	satisfying $x^*_k\in D^*\Phi(x_k,y_k)(\lambda_k)$ for all $k\in\N$.
	Observing that $\{\lambda_k\}_{k\in\N}$ possesses an accumulation point
	$\lambda\in\R^m$, $x^*\in D^*\Phi(\bar x,0)(\lambda)$ follows from the robustness
	of the limiting normal cone and the definition of the coderivative.
	The latter, however, yields $x^*\in\mathcal M(\bar x,0)$, i.e., $\bar x$ is AM-regular.
\end{proof}

Employing the neighborhood characterization of the metric regularity property, 
we now can state a sufficient condition for AM-regularity.

\begin{theorem}\label{thm:metric_regularity_feasibility_map}
	Let $\bar x\in M$ be a feasible point of \eqref{eq:basic_problem} such that $\Phi$ is
	metrically regular at $(\bar x,0)$.
	Then $\bar x$ is AM-regular.
\end{theorem}
\begin{proof}
	Exploiting \cite[Theorem~4.5]{Mordukhovich2006} and the definition of the
	limiting coderivative, metric regularity of $\Phi$ at $(\bar x,0)$
	guarantees the existence of a constant $\kappa>0$ and a neighborhood
	$U$ of $(\bar x,0)$ such that
	\begin{align*}
		&\forall (x,y)\in\gph \Phi
		\cap U\;
		\forall x^*\in\R^n\;\forall \lambda\in\R^m
		\colon\quad
		x^*\in D^*\Phi(x,y)(\lambda)
		\,\Longrightarrow\,\norm{\lambda}\leq\kappa\norm{x^*}.		
	\end{align*}
		
	Choose sequences $\{x_k\}_{k\in\N},\{x^*_k\}_{k\in\N}\subset\R^n$ and
	$\{y_k\}_{k\in\N}\subset\R^m$ as well as $x^*\in\R^n$ such that $x_k\to\bar x$,
	$y_k\to 0$, $x^*_k\to x^*$, and $x_k^*\in\mathcal M(x_k,y_k)$ for all $k\in\N$
	hold. By definition of $\mathcal M$, we find a sequence $\{\lambda_k\}_{k\in\N}\subset\R^m$
	such that $x_k^*\in D^*\Phi(x_k,y_k)(\lambda_k)$ holds for all $k\in\N$. 
	The above considerations show that the sequence $\{\lambda_k\}_{k\in\N}$  
	needs to be bounded since $\{x^*_k\}_{k\in\N}$ is bounded.
	Now, the theorem's assertion follows from \cref{lem:boundedness_of_multipliers}.
\end{proof}

Recall that the Mordukhovich criterion provides a necessary and sufficient condition
for metric regularity of $\Phi$ at arbitrary points of its graph. 
Thus, it can be used as a sufficient
condition for AM-regularity as well.
\begin{corollary}\label{cor:Mordukhovich_criterion_sufficient_for_AMS_regularity}
	Let $\bar x\in M$ be a feasible point of \eqref{eq:basic_problem} such that
	$\ker D^*\Phi(\bar x,0)=\{0\}$. Then $\bar x$ is AM-regular.
\end{corollary}

Clearly, there exist polyhedral set-valued mappings which are not metrically regular
at all points of their graphs. In the light of \cref{thm:polyhedrality_and_AMS_regularity},
this shows that AM-regularity is generally weaker than
metric regularity.

It remains to investigate the relationship between AM-regularity and metric subregularity of
$\Phi$. The subsequently stated example depicts that metric subregularity of $\Phi$ does not
imply validity of AM-regularity.
\begin{example}\label{ex:metric_subregularity_does_not_imply_AMS_regularity}
	We set
	\[
		\forall x\in\R^2\colon\quad
		\Phi(x):=(-x_1^2+x_2,-x_2)-\R^2_-
	\]
	and consider the point $\bar x:=(0,0)$. 
	
	Using the formulas from \cref{sec:geometric_constraints}, we find
	\[
		D^*\Phi(x,y)(\lambda)
		=
		\begin{cases}\{(-2x_1\lambda_1,\lambda_1-\lambda_2)\}	
			&\lambda\in\mathcal N_{\R^2_-}(-x_1^2+x_2-y_1,-x_2-y_2)\\
			\varnothing	&\text{otherwise}
		\end{cases}
	\]
	for all $x,y,\lambda\in\R^2$.
	This yields $\mathcal M(\bar x,(0,0))=\{0\}\times\R$.
	Using the sequences given by
	\[
		\forall k\in\N\colon\quad
		x_{k,1}:=-\tfrac1k\qquad x_{k,2}:=0
		\qquad\qquad
		y_{k,1}:=-\tfrac1{k^2}\qquad y_{k,2}:=0,
	\]
	we have $x_k\to\bar x$ and $y_k\to(0,0)$ as well as 
	$(1,0)\in\mathcal M(x_k,y_k)$ for all $k\in\N$.
	Due to $(1,0)\notin\mathcal M(\bar x,(0,0))$, $\bar x$ is not an
	AM-regular point of the associated constraint set $M$.
	
	One can, however, check that Gfrerer's \emph{second-order sufficient condition for
	metric subregularity} is valid at $\bar x$, see \cite[Corollary~1]{GfrererKlatte2016},
	which shows that $\Phi$ is metrically subregular at $(\bar x,(0,0))$.
\end{example}

The next example depicts that validity of AM-regularity is not enough to ensure
metric subregularity of $\Phi$. Particularly, these conditions are independent of
each other.
\begin{example}\label{ex:AMS_regularity_does_not_imply_metric_subregularity}
	We fix 
	\[
		\forall x\in\R\colon\quad
		\Phi(x):=
		\begin{cases}
			\R	& x\leq 0\\
			[x^2,\infty)	& x>0.
		\end{cases}
	\]
	In this case, we have $M=(-\infty,0]$. Let us consider the point $\bar x:=0$.
	
	Some calculations show
	\[
		D^*\Phi(x,y)(\lambda)
		=
		\begin{cases}
			\{2x\lambda\}	&x>0,\,y=x^2,\,\lambda\geq 0\\
			\R_+			&x=0,\,y\leq 0,\,\lambda=0\\
			\{0\}			&x=y=0,\,\lambda>0\text{ or } 
							x<0,\,\lambda=0\text{ or }
							x\geq 0,\,y>x^2,\,\lambda=0\\
			\varnothing		&\text{otherwise}
		\end{cases}
	\]
	for all $x,y,\lambda\in\R$.
	This yields $\mathcal M(\bar x,0)=\R_+$, and since all other images of the coderivative
	are subsets of $\R_+$, we infer that $\bar x$ is an AM-regular point of $M$.
	
	Setting $x_k:=1/k$ for each $k\in\N$, we find $\dist(x_k,M)=1/k$ and
	$\dist(0,\Phi(x_k))=1/k^2$. Thus, taking the limit $k\to\infty$, it is clear
	that $\Phi$ is not metrically subregular at $(\bar x,0)$.
\end{example}

The proof of the upcoming result, which provides an upper estimate of
the limiting normal cone to the set $M$ at some AM-regular point 
in terms of initial problem data, is inspired by
\cite[proof of Theorem~5.9]{BoergensKanzowMehlitzWachsmuth2019}.
\begin{theorem}\label{thm:representation_of_limiting_normal_cone}
	Let $\bar x\in M$ be a feasible AM-regular point of \eqref{eq:basic_problem}.
	Then we have $\mathcal N_M(\bar x)\subset\mathcal M(\bar x,0)$.
\end{theorem}
\begin{proof}
	Choose $x^*\in\mathcal N_M(\bar x)$ arbitrarily.
	Then we find sequences $\{x_k\}_{k\in\N}\subset M$ and $\{x_k^*\}\subset\R^n$
	such that $x_k\to\bar x$, $x_k^*\to x^*$, and $x_k^*\in\widehat{\mathcal N}_M(x_k)$
	for all $k\in\N$.
	Using the variational description of regular normals from
	\cite[Theorem~1.30(ii)]{Mordukhovich2006}, for each $k\in\N$, we find a differentiable convex
	function $h_k\colon \R^n\to\R$ which achieves a global minimum at $x_k$ when restricted to
	$M$ and which satisfies $\nabla h_k(x_k)=-x_k^*$. 
	Observe that the properties of $h_k$ already guarantee that this function is continuously
	differentiable for each $k\in\N$, see \cite[Corollary of Proposition~2.8]{Phelps1993}. 
	Applying \cref{thm:local_minimizers_AMS_points}
	for fixed $k\in\N$, we find that $x_k$ is an AM-stationary point of the
	optimization problem $\min\{h_k(x)\,|\,x\in M\}$.
	Thus, we find sequences 
	$\{x_{k,\ell}\}_{\ell\in\N},\{\varepsilon_{k,\ell}\}_{\ell\in\N}\subset\R^n$
	and
	$\{y_{k,\ell}\}_{\ell\in\N}\subset\R^m$
	such that $x_{k,\ell}\to x_k$, $\varepsilon_{k,\ell}\to 0$,
	$y_{k,\ell}\to 0$,  as $\ell\to\infty$ and 
	$\varepsilon_{k,\ell}-\nabla h_k(x_{k,\ell})\in\mathcal M(x_{k,\ell},y_{k,\ell})$ 	
	for all $\ell\in\N$ hold.
	We set $x^*_{k,\ell}:=-\nabla h_k(x_{k,\ell})$ for all $\ell\in\N$ and obtain
	$x^*_{k,\ell}\to x^*_k$ as $\ell\to\infty$ by continuous differentiability of $h_k$.
	Exploiting a standard diagonal sequence argument, we, thus, find sequences 
	$\{\bar x_k\}_{k\in\N},\{\varepsilon_k\}_{k\in\N}
	,\{\bar x^*_k\}_{k\in\N}\subset\R^n$, and
	$\{y_k\}_{k\in\N}\subset\R^m$ such that $\bar x_k\to\bar x$, $\varepsilon_k\to 0$, 
	$\bar x^*_k\to x^*$, $y_k\to 0$, and
	$\varepsilon_k+\bar x^*_k\in\mathcal M(\bar x_k,y_k)$ for all $k\in\N$
	hold. Taking the limit $k\to\infty$ and exploiting the fact that $\bar x$ is
	AM-regular, we obtain $x^*\in\mathcal M(\bar x,0)$.
\end{proof}

Let us recall that the assertion of \cref{thm:representation_of_limiting_normal_cone} can be
interpreted in the following way: Observing that $M=\Phi^{-1}(0)$ holds, 
under validity of AM-regularity at a given feasible 
point $\bar x\in M$ of \eqref{eq:basic_problem}, some kind of pre-image rule for the computation
of the limiting normal cone to $M$ in terms of initial problem data, i.e.,
the coderivative of $\Phi$, holds.
Keeping \cite[Proposition~5.3]{Mordukhovich2006} in mind, this alone is enough to show that
AM-regularity is indeed a constraint qualification which guarantees validity of M-stationarity
at the local minimizers of \eqref{eq:basic_problem}.
A similar observation can be made in the presence of metric subregularity of $\Phi$ at 
$(\bar x,0)$, and the upper estimate for the limiting normal cone to $M$ 
can be sharpened if the precise modulus of metric subregularity is known or
can be estimated from above, see \cite[Proposition~4.1]{GfrererOutrata2016b}.

We summarize our results on the relations between all mentioned qualification conditions
in \cref{fig:CQs}.
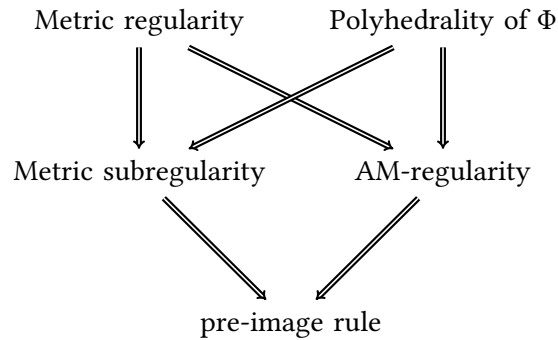
\begin{figure}[h]
\centering
\begin{tikzpicture}[->]

  \node[punkt] at (0,0) 	(A){Metric regularity};
  \node[punkt] at (4,0) 	(B){Polyhedrality of $\Phi$};
  \node[punkt] at (0,-2) 	(C){Metric subregularity};
  \node[punkt] at (4,-2)    (D){AM-regularity};
  \node[punkt] at (2,-4) 	(E){pre-image rule};

  \path     (A) edge[-implies,thick,double] node {}(C)
            (A) edge[-implies,thick,double] node {}(D)
            (B) edge[-implies,thick,double] node {}(C)
            (B) edge[-implies,thick,double] node {}(D)
            (C) edge[-implies,thick,double] node {}(E)
            (D) edge[-implies,thick,double] node {}(E);
\end{tikzpicture}
\caption{
	Relations between constraint qualifications addressing \eqref{eq:basic_problem}
	which guarantee M-stationarity of associated local minimizers.
}
\label{fig:CQs}
\end{figure}

We want to close this section with two remarks.
\begin{remark}\label{rem:generalized_AM_regularity}
	In light of \cref{thm:representation_of_limiting_normal_cone}, it might be
	reasonable to call a set-valued mapping $\Phi\colon\R^n\tto\R^m$ with a closed
	graph AM-regular at $(\bar x,\bar y)\in\gph\Phi$ whenever
	\[
		\limsup\limits_{x\to\bar x,\,y\to\bar y}\mathcal M(x,y)
		\subset
		\mathcal M(\bar x,\bar y)
	\]
	holds. 
	Indeed, this implies validity of the estimate
	\[
		\mathcal N_{\Phi^{-1}(\bar y)}(\bar x)
		\subset
		\mathcal M(\bar x,\bar y).
	\]
	Besides, this approach opens a way to apply AM-regularity in more general
	contexts ranging from applications addressing the limiting variational calculus to
	other interesting settings of optimization theory.
\end{remark}

\begin{remark}\label{rem:normal_semicontinuity}
	Let $\Phi\colon\R^n\tto\R^m$ be a set-valued mapping with a closed graph and fix
	$(\bar x,\bar y)\in\gph\Phi$. In \cite[Definition~5.69]{Mordukhovich2006}, the
	author introduces a concept called \emph{normal semicontinuity} of $\Phi$ at
	$(\bar x,\bar y)$ which demands that
	\[
		\limsup\limits_{x\to\bar x,\,y\to\bar y}\mathcal N_{\Phi(x)}(y)
		\subset
		\mathcal N_{\Phi(\bar x)}(\bar y)
	\]
	holds. This property and related results turned out to be useful in light of,
	e.g., extremal principles for set-valued mappings, multiobjective optimization, and
	optimal control, see \cite[Sections~5.3, 5.5.20, and 6.1.5]{Mordukhovich2006}.
	
	By definition, normal semicontinuity seems to be related to AM-regularity.
	However, one can easily check that both properties are independent of each other.
	Indeed, considering the data from \cref{ex:non_AM_stationary_point}, we find
	that AM-regularity fails at $\bar x:=0$ while the mapping $\Phi$ is normally
	semicontinuous at $(\bar x,0)$. On the other hand, let us revisit
	\cref{ex:AMS_regularity_does_not_imply_metric_subregularity}. 
	Therein, $\bar x:=0$ is AM-regular while $\Phi$ is not normally semicontinuous
	at $(\bar x,0)$.	
	Nevertheless, AM-regularity and normal semicontinuity
	both provide \emph{sequential stability properties} of set-valued mappings
	which is why we believe that AM-regularity could be useful in similar settings
	where normal semicontinuity turned out to be valuable.
\end{remark}

\section{Decoupling of abstract constraints}\label{sec:decoupling_abstract_constraints}

In this section, we want to investigate the particular case where the
mapping $\Phi$ is given by
\begin{equation}\label{eq:abstract_constraints}
	\forall x\in \R^n\colon\quad
	\Phi(x):=\begin{pmatrix}\Gamma(x)\\x-C\end{pmatrix}
\end{equation}
where $\Gamma\colon \R^n\tto\R^\ell$ is a set-valued mapping with closed graph
and $C\subset \R^n$ is a nonempty, closed
set. Roughly speaking, we assume that the abstract constraint set $C$ is
\emph{simple} and shall be decoupled from the more enhanced constraints which are
modeled with the aid of the generalized equation $0\in\Gamma(x)$.

Exploiting the product rule for coderivative calculus from \cref{lem:product_rule},
a feasible point $\bar x\in M$ of problem \eqref{eq:basic_problem} where $\Phi$
is given as in \eqref{eq:abstract_constraints} is M-stationary if and only if 
there is a multiplier $\tilde\lambda\in\R^\ell$ satisfying
\[
	0\in\partial f(\bar x)+D^*\Gamma(\bar x,0)(\tilde\lambda)+\mathcal N_C(\bar x).
\]
Furthermore, applying \cref{def:Asymptotic_M_stationary_point} to the situation at hand,
$\bar x$ is an AM-stationary point of the associated 
problem \eqref{eq:basic_problem} if and only if there exist sequences 
$\{x_k\}_{k\in\N},\{\varepsilon_k\}_{k\in\N},\{z_k\}_{k\in\N}\subset \R^n$ as well as
$\{\tilde y_k\}_{k\in\N},\{\tilde\lambda_k\}_{k\in\N}\subset\R^\ell$ 
satisfying $x_k\to\bar x$, $\varepsilon_k\to 0$, $\tilde y_k\to 0$, $z_k\to 0$, 
and 
\[
	\forall k\in\N\colon\quad
	\varepsilon_k\in\partial f(x_k)+D^*\Gamma(x_k,\tilde y_k)(\tilde\lambda_k)+\mathcal N_C(x_k-z_k).
\]
The relation $z_k:=0$ for all $k\in\N$ seems to be desirable since this would mean that
all points $x_k$ from above already satisfy the abstract constraint $x\in C$ hidden in
the definition of $\Phi$, i.e., some \emph{partial} feasibility of the sequence $\{x_k\}_{k\in\N}$
would be guaranteed in this situation. This motivates the subsequent definition of \emph{decoupled}
AM-stationary points.
\begin{definition}\label{def:decoupled_AMS_point}
	Let $\Phi$ be given as in \eqref{eq:abstract_constraints}.
	A feasible point $\bar x\in M$ of the associated problem \eqref{eq:basic_problem}
	is referred to as a \emph{decoupled asymptotically Mordukhovich-stationary point} 
	(dAM-stationary point
	for short) whenever there are sequences 
	$\{x_k\}_{k\in\N},\{\varepsilon_k\}_{k\in\N}\subset \R^n$ as well as
	$\{\tilde y_k\}_{k\in\N},\{\tilde\lambda_k\}_{k\in\N}\subset\R^\ell$ satisfying 
	$x_k\to\bar x$, $\varepsilon_k\to 0$,
	$\tilde y_k\to 0$,
	and 
	\[
		\forall k\in\N\colon\quad
		\varepsilon_k\in\partial f(x_k)+D^*\Gamma(x_k,\tilde y_k)(\tilde\lambda_k)+\mathcal N_C(x_k).
	\]
\end{definition}

Let us note that, by definition, the sequence $\{x_k\}_{k\in\N}\subset\R^n$ 
from \cref{def:decoupled_AMS_point}
has to satisfy $\{x_k\}_{k\in\N}\subset\dom\Gamma\cap C$.

The following theorem shows that under an additional assumption on the mapping $\Gamma$,
each local minimizer of \eqref{eq:basic_problem} with $\Phi$ given as in \eqref{eq:abstract_constraints}
is already a dAM-stationary point.

\begin{theorem}\label{thm:d_AMS_points_via_Aubin_property}
	Let $\bar x\in M$ be a local minimizer of \eqref{eq:basic_problem} where $\Phi$
	is given as in \eqref{eq:abstract_constraints}.
	Furthermore, assume that $\Gamma$ possesses the Aubin property at $(\bar x,0)$.
	Then $\bar x$ is a dAM-stationary point of \eqref{eq:basic_problem}.
\end{theorem}
\begin{proof}
	To start, observe that by definition of the Aubin property, we find 
	$\gamma,\delta>0$ such that $\Gamma$ possesses the Aubin property at all the points
	from $\gph\Gamma\cap(\mathbb B_\gamma(\bar x)\times\mathbb B_\delta(0))$.
	For small enough $\gamma,\delta>0$, we can also guarantee
	$\Gamma(x)\cap\mathbb B_{\delta/2}(0)\neq\varnothing$ 
	for all $x\in\mathbb B_\gamma(\bar x)$ since $\Gamma$ is inner semicontinuous
	at $(\bar x,0)$.
	
	Next, for each $k\in\N$, we investigate the optimization problem
	\begin{equation}\label{eq:penalized_problem_2}\tag{Q$(k)$}
		\begin{aligned}
			f(x)+\frac{k}{2}\left(\rho_\Gamma(x,y)+\norm{y}^2\right)+\frac12\norm{x-\bar x}^2
				&\,\to\,\min\limits_{x,y}\\
				x&\,\in\,C\cap\mathbb B_\gamma(\bar x)\\
				y&\,\in\,\mathbb B_{\delta/4}(0).
		\end{aligned}
	\end{equation}
	Observing that the objective function of this problem is lower semicontinuous by
	\cref{lem:lower_semicontinuity_of_generalized_distance_function} while its feasible set
	is nonempty and compact, \eqref{eq:penalized_problem_2} possesses a global minimizer
	$(x_k,y_k)\in \R^n\times\R^\ell$ for each $k\in\N$. Since
	$\{x_k\}_{k\in\N}$ and $\{ y_k\}_{k\in\N}$ are bounded, we may pass
	to a subsequence (without relabeling) in order to find $\tilde x\in \mathbb B_\gamma(\bar x)$ and
	$\tilde y\in\mathbb B_{\delta/4}(0)$ such that $x_k\to\tilde x$ and $ y_k\to\tilde y$.
	Similar as in the proof of \cref{thm:local_minimizers_AMS_points}, we find $\tilde y=0$.
	In analogous way, we obtain $0\in\Gamma(\tilde x)$ by lower semicontinuity of
	the generalized distance function. Finally, the closedness of $C$ guarantees $\tilde x\in C$, i.e.,
	$\tilde x\in M$. Furthermore, $\tilde x=\bar x$ can be shown as in the proof of 
	\cref{thm:local_minimizers_AMS_points}.
	
	For fixed $k\in\N$, we know $\Gamma(x_k)\cap\mathbb B_{\delta/2}(0)\neq\varnothing$
	from the choice of $\gamma$ and $\delta$.
	Due to $y_k\in\mathbb B_{\delta/4}(0)$, we have 
	$\varnothing\neq\Pi(y_k,\Gamma(x_k))\subset\mathbb B_\delta(0)$. Particularly, $\Gamma$
	possesses the Aubin property at all point from $\{x_k\}\times\Pi(y_k,\Gamma(x_k))$.
	As a consequence, \cref{lem:generalized_distance_function} guarantees that $\rho_\Gamma$
	is locally Lipschitz continuous at $(x_k,y_k)$.
	For sufficiently large $k\in\N$, $x_k$ is an interior point of $\mathbb B_\gamma(\bar x)$
	while $y_k$ is an interior point of $\mathbb B_{\delta/4}(0)$.
	Due to these observations, we may now apply \cite[Proposition~5.3]{Mordukhovich2006},
	the sum rule for the limiting subdifferential, see
	\cite[Theorem~3.36]{Mordukhovich2006}, and \cref{lem:generalized_distance_function}
	in order to find $\tilde y_k\in\Pi(y_k,\Gamma(x_k))$ such that
	\[
		(0,0)
		\in  
		\partial f(x_k)\times\{0\}
		+
		\mathcal N_{\gph\Gamma}(x_k,\tilde y_k)
		+
		\{(x_k-\bar x,ky_k)\}
		+
		\mathcal N_C(x_k)\times\{0\}
	\]
	holds for large enough $k\in\N$. 
	Defining $\tilde\lambda_k:=ky_k$ and $\varepsilon_k:=\bar x-x_k$, this yields
	\[
		\varepsilon_k\in\partial f(x_k)+D^*\Gamma(x_k,\tilde y_k)(\tilde\lambda_k)+\mathcal N_C(x_k)
	\]
	for large enough $k\in\N$.
	The above observations guarantee $\varepsilon_k\to 0$.
	It remains to show $\tilde y_k\to 0$ in order to complete the proof.
	Since $\Gamma$ is inner semicontinuous at $(\bar x,0)$, we find a sequence
	$\{\bar y_k\}_{k\in\N}\subset\R^\ell$ with $\bar y_k\to 0$ and $\bar y_k\in\Gamma(x_k)$ for
	all sufficiently large $k\in\N$. By definition of the projector, we have
	\[
		\norm{\tilde y_k}
		\leq
		\norm{\tilde y_k-y_k}+\norm{y_k}
		\leq
		\norm{\bar y_k-y_k}+\norm{y_k}
		\leq 
		2\norm{y_k}+\norm{\bar y_k}
		\to
		0,
	\]
	and this, finally, shows $\tilde y_k\to 0$.
\end{proof}

The subsequently stated example demonstrates that the statement of
\cref{thm:d_AMS_points_via_Aubin_property} does not remain true in
general when $\Gamma$ does not possess the Aubin property at the
point of interest.
\begin{example}\label{ex:non_d_AMS_local_minimizer}
	We investigate the set-valued mapping
	$\Gamma\colon\R^2\tto\R$ given by
	\[
		\forall x\in\R^2\colon\quad
		\Gamma(x)
		:=
		\begin{cases}
			[0,\infty)	&	x_2=0\\
			\varnothing	&	\text{otherwise}
		\end{cases}
	\]
	as well as the closed set $C:=\{x\in\R^2\,|\,x_1^2+(x_2-1)^2\leq 1\}$.
	We consider the associated optimization problem
	\[
		\begin{aligned}
			x_1&\,\to\,\min\\
			0&\,\in\,\Gamma(x)\\
			x&\,\in\,C.
		\end{aligned}
	\]
	Its uniquely determined feasible point
	and, thus, global minimizer is $\bar x:=(0,0)$.
	
	Assuming that $\bar x$ is a dAM-stationary point of the problem of interest and keeping
	the relation $\dom\Gamma\cap C=\{\bar x\}$ in mind, there need
	to exist sequences $\{\varepsilon_k\}_{k\in\N}\subset\R^2$, $\{\tilde y_k\}_{k\in\N}\subset\R$,
	and $\{\tilde \lambda_k\}_{k\in\N}\subset\R$ such that $\varepsilon_k\to (0,0)$ and $\tilde y_k\to 0$
	as well as
	\begin{equation}\label{eq:d_AMS_condition_example}
		\forall k\in\N\colon\quad
		(\varepsilon_{k,1},\varepsilon_{k,2})
		\in 
		(1,0)+D^*\Gamma(\bar x,\tilde y_k)(\tilde \lambda_k)+\{0\}\times\R_-
	\end{equation}
	hold true. We note, however, that
	\[
		D^*\Gamma(\bar x,\tilde y)(\tilde \lambda)=
		\begin{cases}
			\{0\}\times\R	&	\tilde y=0,\,\tilde \lambda\geq 0
								\;\text{or}\;
								\tilde y>0,\,\tilde\lambda=0\\
			\varnothing		&	\text{otherwise}			
		\end{cases}
	\]
	is valid for all $\tilde y,\tilde \lambda\in\R$. 
	Thus, \eqref{eq:d_AMS_condition_example} yields
	$\varepsilon_{k,1}=1$ for all $k\in\N$ which contradicts $\varepsilon_k\to (0,0)$.
	Thus, $\bar x$ is not a dAM-stationary point of the problem of interest.
	
	Note that $\Phi$ is a polyhedral set-valued mapping which does not
	possess the Aubin property at the reference point $(\bar x,0)$.
\end{example}
	
Keeping our arguments from \cref{sec:asymptotic_regularity} in mind,
\cref{thm:d_AMS_points_via_Aubin_property} motivates the definition
of another constraint qualification weaker than AM-regularity which
ensures that a dAM-stationary point of \eqref{eq:basic_problem} where
$\Phi$ is given as in \eqref{eq:abstract_constraints} is already an M-stationary point.
For that purpose, we introduce a set-valued mapping 
$\widetilde{\mathcal M}\colon\R^n\times\R^\ell\tto\R^n$ by
\begin{equation}\label{eq:definition_tilde_M}
	\forall x\in\R^n\,\forall \tilde y\in\R^\ell\colon\quad
	\widetilde{\mathcal M}(x,\tilde y):=\bigcup\limits_{\tilde\lambda\in\R^\ell}
	D^*\Gamma(x,\tilde y)(\tilde\lambda)+\mathcal N_C(x).
\end{equation}

\begin{definition}\label{def:d_AMS_regularity}
	A feasible point $\bar x\in M$ of \eqref{eq:basic_problem} where $\Phi$ is given as
	in \eqref{eq:abstract_constraints} is said to be 
	\emph{decoupled asymptotically Mordukhovich-regular} (dAM-regular for short) whenever
	\[
		\limsup\limits_{x\to\bar x,\,\tilde y\to 0}
		\widetilde{\mathcal M}(x,\tilde y)
		\subset
		\widetilde{\mathcal M}(\bar x,0)
	\]
	is valid.
\end{definition}

For each feasible point $\bar x\in M$ of \eqref{eq:basic_problem} for $\Phi$ given in \eqref{eq:abstract_constraints}, we have the relations 
$\mathcal M(\bar x,(0,0))=\widetilde{\mathcal M}(\bar x,0)$ and
\[
	\limsup\limits_{x\to\bar x,\,\tilde y\to 0}
	\widetilde{\mathcal M}(x,\tilde y)
	\subset
	\limsup\limits_{x\to\bar x,\,\tilde y\to 0,\,z\to 0}
	\mathcal M(x,(\tilde y,z))
\]
which is why dAM-regularity is weaker than AM-regularity as
promoted above. \cref{ex:non_d_AMS_local_minimizer} shows that there are
situations where dAM-regularity is strictly weaker than AM-regularity.
Therein, $\bar x$ is a dAM-regular point. On the other hand,
we have $(1,0)\in\limsup_{x\to\bar x,\,\tilde y\to 0,\,z\to 0}\mathcal M(x,(\tilde y,z))$
but $\mathcal M(\bar x,(0,0))=\{0\}\times\R$ which is why $\bar x$ cannot
be AM-regular. The subsequent example visualizes that dAM-regularity might be
strictly weaker than AM-regularity even in situations where $\Gamma$ possesses the
Aubin property at all points of its graph.
\begin{example}\label{ex:dAM_regularity_weaker_than_AM_regularity}
	We set $C:=\R_+$ as well as
	\[
		\forall x\in\R\colon\quad \Gamma(x):=[-x^2,\infty)
	\]
	and investigate the mapping $\Phi$ from \eqref{eq:abstract_constraints}.
	In this situation, $M=\R_+$ is valid. Let us focus on the point $\bar x:=0$.
	In \cref{ex:non_AM_stationary_point}, one can find a formula for the
	coderivative of $\Gamma$. Using it and exploiting $\mathcal N_C(0)=\R_-$, we find
	\[
		\widetilde{\mathcal M}(x,\tilde y)
		=
		\begin{cases}
			\R_-	&x>0,\,\tilde y=-x^2\;\text{or}\;x=0,\,\tilde y\geq 0\\
			\{0\}	&x>0,\,\tilde y>-x^2\\
			\varnothing	&\text{otherwise}
		\end{cases}
	\]
	for all $x,\tilde y\in\R$. Due to $\widetilde{\mathcal M}(\bar x,0)=\R_-$, $\bar x$ is
	dAM-regular.
	Let us set
	\[
		x_k:=-\tfrac{1}{k}
		\qquad
		\tilde y_k:=-\tfrac{1}{k^2}
		\qquad
		z_k:=-\tfrac{1}{k}
	\]
	for all $k\in\N$. Then we find $\mathcal M(x_k,(\tilde y_k,z_k))=\R$, i.e.,
	\[
		\limsup\limits_{x\to\bar x,\,\tilde y\to 0,\,z\to 0}\mathcal M(x,(\tilde y,z))
		=
		\R,
	\]
	and this shows that $\bar x$ cannot be AM-regular since 
	$\mathcal M(\bar x,(0,0))=\widetilde{\mathcal M}(\bar x,0)=\R_-$ holds true.
	Observing that $\Gamma$ is the sum of the locally Lipschitzian single-valued mapping
	$x\mapsto -x^2$ and the constant set $\R_+$, $\Gamma$ possesses the Aubin property
	at all points of its graph.
\end{example}

Clearly, $\bar x\in M$ is an M-stationary point of \eqref{eq:basic_problem} 
where $\Phi$ is given as in \eqref{eq:abstract_constraints} 
if and only if
\[
	\partial f(\bar x)\cap\bigl(-\widetilde{\mathcal M}(\bar x,0)\bigr)\neq\varnothing
\]
is valid. Furthermore, similar as in the proof of \cref{lem:chararcterizing_AMS_points_via_M},
one can show that whenever $\bar x$ is a dAM-stationary point of the problem of interest, then
\[
	\partial f(\bar x)\cap
	\left(
		-\limsup\limits_{x\to\bar x,\,\tilde y\to 0}
		\widetilde{\mathcal M}(x,\tilde y)
	\right)\neq\varnothing
\]
is true. Thus, \cref{thm:d_AMS_points_via_Aubin_property} yields the following result.
\begin{theorem}\label{thm:d_AMS_regularity_and_Aubin_property_CQ}
	Let $\bar x\in M$ be a dAM-regular local minimizer of \eqref{eq:basic_problem}
	where $\Phi$ is given as in \eqref{eq:abstract_constraints}.
	Furthermore, let $\Gamma$ possess the Aubin property at $(\bar x,0)$.
	Then $\bar x$ is an M-stationary point of \eqref{eq:basic_problem}.
\end{theorem}

Using the product rule from \cref{lem:product_rule} as well as the result of
\cref{cor:Mordukhovich_criterion_sufficient_for_AMS_regularity}, the
condition
\begin{equation}\label{eq:Mordukhovich_criterion_decoupled_case}
	0\in D^*\Gamma(\bar x,0)(y^*)+z^*,\,z^*\in\mathcal N_C(\bar x)
	\quad\Longrightarrow\quad
	y^*=0,\,z^*=0
\end{equation}
is sufficient for AM-regularity and, thus, dAM-regularity of a feasible
point $\bar x\in M$ of \eqref{eq:basic_problem} where $\Phi$ is given as in \eqref{eq:abstract_constraints}.

Finally, we would like to mention that the assertion of 
\cref{thm:representation_of_limiting_normal_cone} also holds true in the
present setting under validity of dAM-regularity whenever $\Gamma$ possesses
the Aubin property at the point of interest.
\begin{theorem}\label{thm:representation_limiting_normal_cone_d_AMS_regularity}
	Let $\bar x\in M$ be a feasible dAM-regular point of \eqref{eq:basic_problem}
	where $\Phi$ is given as in \eqref{eq:abstract_constraints}.
	Furthermore, let $\Gamma$ possess the Aubin property at $(\bar x,0)$.
	Then we have $\mathcal N_M(\bar x)\subset\widetilde{\mathcal M}(\bar x,0)$.
\end{theorem}
\begin{proof}
	The proof is analogous to the one of \cref{thm:representation_of_limiting_normal_cone}
	exploiting \cref{thm:d_AMS_points_via_Aubin_property} and the fact that $\Gamma$
	possesses the Aubin property at all points from $\gph\Gamma\cap U$ where 
	$U\subset\R^n\times\R^\ell$ is a sufficiently small neighborhood of
	$(\bar x,0)$.
\end{proof}

\section{Applications of asymptotic regularity}\label{sec:applications}
\subsection{Asymptotic regularity for mathematical programs with geometric constraints}\label{sec:geometric_constraints}

In this section, we assume that $\Phi\colon \R^n\rightrightarrows \R^\ell\times\R^n$ is given by 
\begin{equation}\label{eq:geometric_constraints}
	\forall x\in \R^n\colon\quad 
	\Phi(x)
	:=
		\begin{pmatrix}
			G(x)-K\\
			x-C
		\end{pmatrix}
\end{equation}
where $G\colon \R^n\rightarrow\R^\ell$ is a locally Lipschitz continuous, single-valued mapping, 
while the sets $K\subset\R^\ell$  and $C\subset \R^n$ are nonempty as well as closed. 
Thus, the associated feasible region of
\eqref{eq:basic_problem} is given by 
\begin{equation}\label{eq:standard_preimage_constraints}
	M=\{x\in C\,|\,G(x)\in K\},
\end{equation}
and this rather general description still covers numerous interesting classes
of optimization problems comprising standard nonlinear problems, instances
of conic programming, disjunctive programs (e.g., mathematical problems with
complementarity, vanishing, switching, or cardinality constraints, see
\cref{sec:disjunctive_programming}), and
conic complementarity programming. Generally, one refers to constraint systems
of this type as \emph{geometric constraints}.

We observe that the structure of $\Phi$ is precisely the one discussed in
\cref{sec:decoupling_abstract_constraints} if we use
the feasibility mapping $\Gamma\colon \R^n\tto\R^\ell$ given by $\Gamma(x):=G(x)-K$
for all $x\in \R^n$.
Observing that $G$ is a locally Lipschitz
continuous map, $\Gamma$ possesses the
Aubin property at each point of its graph.
Due to \cref{thm:d_AMS_points_via_Aubin_property,thm:d_AMS_regularity_and_Aubin_property_CQ}, the
local minimizers of the underlying optimization problem are always dAM-stationary points
and dAM-regularity provides a constraint qualification for the presence of
M-stationarity.
Using the coderivative sum rule from 
\cite[Theorem~1.62]{Mordukhovich2006}, we have
\[
		\forall (x,\tilde y)\in\gph\Gamma\;\forall\tilde\lambda\in \R^\ell
		\colon\quad
		D^*\Gamma(x,\tilde y)(\tilde\lambda)
		=
		\begin{cases}
			D^*G(x)(\tilde\lambda)	&\tilde\lambda\in\mathcal N_K(G(x)-\tilde y)\\
			\varnothing			&\text{otherwise.}
		\end{cases}
\]
Particularly, the mapping $\widetilde{\mathcal M}\colon \R^n\times\R^\ell\tto \R^n$ 
from \eqref{eq:definition_tilde_M} takes the form
\[
	\forall x\in\R^n\,\forall \tilde y\in\R^\ell\colon\quad
	\widetilde{\mathcal M}(x,\tilde y)=D^*G(x)\mathcal N_K(G(x)-\tilde y)+\mathcal N_C(x).
\]
The latter can be used to specify the precise nature of dAM-regularity
for particular classes of optimization problems with geometric constraints.
We also note that validity of this constraint qualification at an
arbitrary point $\bar x\in M$ yields the estimate
\[
	\mathcal N_M(\bar x)
	\subset
	D^*G(\bar x)\mathcal N_K(G(\bar x))+\mathcal N_C(\bar x),
\]
see \cref{thm:representation_limiting_normal_cone_d_AMS_regularity}.
In the literature, metric subregularity of $\Phi$ from \eqref{eq:geometric_constraints} 
at $(\bar x,(0,0))$ is often assumed for that purpose, see e.g.\
\cite[Theorem~4.1]{HenrionJouraniOutrata2002} where it is shown that
already metric subregularity of $\widehat\Phi\colon\R^n\tto\R^\ell$ given by
\[
	\forall x\in\R^n\colon\quad
	\widehat\Phi(x):=
	\begin{cases}
		G(x)-K&x\in C\\\varnothing&\text{otherwise}
	\end{cases}
\]
at the point $(\bar x,0)$ is enough for that purpose. 
In the light of
\cref{sec:asymptotic_regularity}, dAM-regularity is, however, independent of 
the metric subregularity of $\Phi$ and, thus, provides a different approach
to this pre-image rule. 
In case where $G$ is smooth, $C=\R^n$, and $K$ is of special structure,
the fact that dAM-regularity provides a constraint
qualification has been observed in \cite[Theorem~3.13]{Ramos2019}.
A related observation has been made in the context of semidefinite
programming in \cite{AndreaniHaeserViana2020}.
Replacing the image space $\R^\ell$ by the Hilbert space of all real symmetric 
matrices, this paper's theory covers this special situation, too.
Under additional assumptions on the data (e.g., convexity of $K$ and $C$),
related results can be obtained for optimization problems in Banach spaces as well,
see \cite{BoergensKanzowMehlitzWachsmuth2019} and \cref{rem:convexity} below.
It follows from \cite[Section~4]{Ramos2019} that dAM-regularity for feasible sets
of type \eqref{eq:standard_preimage_constraints}
is not related to suitable notions of \emph{pseudo-} and
\emph{quasinormality} which apply to feasible sets of type 
\eqref{eq:standard_preimage_constraints}, see 
\cite[Definition~3.4]{BenkoCervinkaHoheisel2019} and \cite[Definition~4.2]{GuoYeZhang2013}
as well.
On the other hand, we know from our investigations in the earlier sections that
this new constraint qualification is generally weaker than metric regularity of
$\Phi$ from \eqref{eq:geometric_constraints} at some point $(\bar x,(0,0))\in\gph\Phi$, and
the latter is equivalent to
\[
	-G'(\bar x)^\top\tilde\lambda\in\mathcal N_C(\bar x),\,\tilde\lambda\in\mathcal N_K(G(\bar x))
	\quad\Longrightarrow\quad
	\tilde\lambda=0
\]
in case where $G$ is continuously differentiable at $\bar x$, see
\eqref{eq:Mordukhovich_criterion_decoupled_case} as well. This condition
is well known as \emph{no nonzero abnormal multiplier constraint qualification} (NNAMCQ)
or \emph{generalized Mangasarian--Fromovitz constraint qualification} (GMFCQ) in the
literature.

In the subsequently stated example, we interrelate our findings with the results from
\cite{AndreaniMartinezRamosSilva2016} where a sequential constraint qualification has
been introduced for standard nonlinear programs.
\begin{example}\label{ex:recovering_CCP}
	Fix $\ell:=p+q$, $K:=\R_-^p\times\{0\}$, as well as $C:=\R^n$ and
	let the mapping $G\colon\R^n\to\R^{p+q}$ be continuously differentiable.
	Furthermore, let $G_1,\ldots,G_{p+q}\colon\R^n\to\R$ be the component mappings
	associated with $G$. In this situation, the mapping $\widetilde{\mathcal M}$ from
	above takes the particular form
	\[
		\widetilde{\mathcal M}(x,\tilde y)
		=
		\left\{
			\sum_{i=1}^{p+q}\tilde\lambda_i\nabla G_i(x)
			\,\middle|\,
			\begin{aligned}
				\min(\tilde\lambda_i,\tilde y_i-G_i(x))&=0&&\forall i\in\{1,\ldots,p\}\\
				\tilde y_i-G_i(x)&=0&&\forall i\in\{p+1,\ldots,q\}
			\end{aligned}
		\right\}
	\]
	for all $x\in\R^n$ and $\tilde y\in\R^{p+q}$.
	We want to compare the associated AM-regularity condition (which equals dAM-regularity
	due to $C=\R^n$) with the so-called \emph{cone-continuity property} (CCP for short) from 
	\cite[Definition~3.1]{AndreaniMartinezRamosSilva2016} which has been shown to be a
	constraint qualification for standard nonlinear problems. It is based on 
	the mapping $\mathcal K\colon\R^n\tto\R^n$ given by
	\[
		\forall x\in\R^n\colon\quad 
		\mathcal K(x):=
		\left\{
			\sum\limits_{i=1}^{p+q}\lambda_i\nabla G_i(x)\,\middle|\,
			\min(\lambda_i,-G_i(\bar x))=0\quad\forall i\in\{1,\ldots,p\}
		\right\}
	\]
	and demands that
	\[
		\limsup\limits_{x\to\bar x}\mathcal K(x)\subset\mathcal K(\bar x)
	\]
	holds at a given point $\bar x\in M$. 
	Note that we have $\widetilde{\mathcal M}(\bar x,0)=\mathcal K(\bar x)$. 
	
	Observing that, for each $x\in\R^n$, 
	we have $\mathcal K(x)=\widetilde{\mathcal M}(x,G(x)-G(\bar x))$ 
	while the convergence $G(x)-G(\bar x)\to 0$ holds as $x\to\bar x$, 
	validity of AM-regularity at $\bar x$
	yields that CCP holds at $\bar x$, too. 
	On the other hand, let CCP hold at $\bar x$.
	If $\{x_k\}_{k\in\N},\{x_k^*\}_{k\in\N}\subset\R^n$ as well as 
	$\{\tilde y_k\}_{k\in\N}\subset\R^{p+q}$
	are sequences with $x_k\to\bar x$, $x_k^*\to x^*$ for some $x^*\in\R^n$, and
	$\tilde y_k\to 0$ such that $x_k^*\in\widetilde{\mathcal M}(x_k,\tilde y_k)$ holds for each $k\in\N$,
	then we find a sequence $\{\tilde\lambda_k\}_{k\in\N}\subset\R^{p+q}$ such that
	$x_k^*=\sum_{i=1}^{p+q}\tilde\lambda_{k,i}\nabla G_i(x_k)$ and
	$\min(\tilde\lambda_{k,i},\tilde y_{k,i}-G_i(x_k))=0$, $i=1,\ldots,p$, are valid for all $k\in\N$.
	Let $I(\bar x):=\{i\in\{1,\ldots,p\}\,|\,G_i(\bar x)=0\}$ denote the set of indices
	associated with inequality constraints active at $\bar x$. 
	For each $k\in\N$, we have $\tilde\lambda_{k,i}\geq 0$ for all $i\in\{1,\ldots,p\}$.
	Whenever $i\notin I(\bar x)$ holds, $\tilde y_{k,i}-G_i(x_k)>0$ is valid for sufficiently large
	$k\in\N$ due to $x_k\to\bar x$, $\tilde y_k\to 0$, and continuity of $G$. 
	Thus, we have $\tilde\lambda_{k,i}=0$ for
	sufficiently large $k\in \N$ and all $i\notin I(\bar x)$. 
	This particularly yields $\min(\tilde\lambda_{k,i},-G_i(\bar x))=0$ for sufficiently large
	$k\in\N$ and all $i\in\{1,\ldots,p\}$. 
	Hence, we have shown $x^*_k\in\mathcal K(x_k)$ for sufficiently large
	$k\in\N$. By validity of CCP, $x^*\in\mathcal K(\bar x)=\widetilde{\mathcal M}(\bar x,0)$
	follows, i.e., $\bar x$ is AM-regular.
	
	The above investigations show that AM-regularity is equivalent to CCP in the setting of
	standard nonlinear programming. Let us mention that CCP has also been referred to as
	AKKT-regularity in the literature which is why the latter is a particular instance
	of AM-regularity as well.
\end{example}

In the subsequent remark, we address the situation where $K\subset\R^\ell$
is convex and $G$ is continuously differentiable.
\begin{remark}\label{rem:convexity}
	Assume that $K\subset\R^\ell$ is convex while $G$ is
	continuously differentiable.
	Adapting the proof of \cite[Proposition~3.3]{BoergensKanzowMehlitzWachsmuth2019},
	whenever $\bar x\in M$ is a local minimizer of the associated problem
	\eqref{eq:basic_problem} where $\Phi$ is given as in \eqref{eq:geometric_constraints},
	we find sequences $\{x_k\}_{k\in\N}\subset\R^n$ and $\{\varepsilon_k\}_{k\in\N}\subset\R^n$
	such that $x_k\to\bar x$, $\varepsilon_k\to 0$, and
	\[
		\forall k\in\N\colon\quad
		\varepsilon_k\in\partial f(x_k)+G'(x_k)^\top(K-G(x_k))^\circ+\mathcal N_C(x_k)
	\]
	hold. Note, however, that we cannot replace $(K-G(x_k))^\circ$ by $\mathcal N_K(G(x_k))$ in the
	above formula since $G(x_k)$ does not need to be an element of $K$ in general.
	Consequently, this sequential concept of stationarity is slightly different from 
	dAM-stationarity. However, it can be used in similar fashion for the derivation of a
	constraint qualification which guarantees that $\bar x$ satisfies
	the M-stationarity conditions of the associated optimization problem, namely
	\[
		\limsup\limits_{x\to\bar x}\widehat{\mathcal M}(x)
		\subset
		\widehat{\mathcal M}(\bar x)
	\]
	where $\widehat{\mathcal M}\colon\R^n\tto\R^n$ is defined by
	\[
		\forall x\in\R^n\colon\quad
		\widehat{\mathcal M}(x):=G'(x)^\top(K-G(x))^\circ+\mathcal N_C(x),
	\]
	see \cite[Corollary~4.8]{BoergensKanzowMehlitzWachsmuth2019} as well.
\end{remark}

\subsection{Asymptotic regularity in disjunctive programming}
\label{sec:disjunctive_programming}

In this section, we take a closer look at
\emph{mathematical programs with disjunctive constraints}
which are optimization problems of the form
\begin{equation}\label{eq:disjunctive_program}\tag{MPDC}
	\begin{aligned}
		f(x)&\,\to\,\min\\
		G(x)&\,\in\,K
	\end{aligned}
\end{equation}
where $f\colon\R^n\to\R$ and
$G\colon\R^n\to\R^m$ are continuously differentiable mappings and
$K:=\bigcup_{i=1}^pD_i$ is the union of finitely many convex polyhedral sets
$D_1,\ldots,D_p\subset\R^m$. 
Again, we denote its feasible set by $M$.
Such optimization problems have been dealt with
e.g.\ in
\cite{BenkoCervinkaHoheisel2019,BenkoGfrerer2018,FlegelKanzowOutrata2007,Gfrerer2014,Mehlitz2019b}
in terms of first- and second-order optimality conditions as well as suitable constraint
qualifications. The model \eqref{eq:disjunctive_program} is attractive since it
covers numerous classes from structured nonlinear optimization like
mathematical programs with complementarity constraints (MPCCs), mathematical programs with
vanishing constraints (MPVCs), mathematical programs with switching constraints (MPSCs), or
cardinality-constrained mathematical problems (CCMPs), see \cite[Section~5]{Mehlitz2019b}
for an overview of these popular classes from disjunctive programming and references to the literature.
Noting that \eqref{eq:disjunctive_program} is a particular instance of a mathematical program
with geometric constraints, we are in position to apply the theory from above
to the problem of interest. Noting that no abstract constraints are present in
the formulation of \eqref{eq:disjunctive_program}, we rely on AM-regularity as a constraint
qualification for \eqref{eq:basic_problem}. 
The associated mapping $\mathcal M\colon\R^n\times\R^m\tto\R^n$ is given by
\[
	\forall x\in\R^n\,\forall y\in\R^m\colon\quad
	\mathcal M(x,y)=G'(x)^\top\mathcal N_K(G(x)-y)
\]
in this setting.

Our first result, which is inspired by our observations from \cref{ex:recovering_CCP},
shows that we can rely on the continuity properties of a much simpler map than $\mathcal M$ 
in order to check validity of AM-regularity. 
The proof of this result exploits some arguments we already used to verify
\cref{thm:polyhedrality_and_AMS_regularity}.
\begin{theorem}\label{thm:AM_regularity_for_disjunctive_programs}
	Fix a feasible point $\bar x\in M$ of \eqref{eq:disjunctive_program} and define
	a set-valued mapping $\mathcal K\colon\R^n\tto\R^n$ by means of
	\[
		\forall x\in\R^n\colon\quad 
		\mathcal K(x):=G'(x)^\top\mathcal N_K(G(\bar x)).
	\]
	Then $\bar x$ is AM-regular if and only if the following condition holds:
	\begin{equation}\label{eq:AM_regularity_for_disjunctive_programs}
		\limsup\limits_{x\to\bar x}\mathcal K(x)\subset\mathcal K(\bar x).
	\end{equation}
\end{theorem}
\begin{proof}
	We show both implications separately.\\
	$[\Longrightarrow]$ Let $\bar x$ be AM-regular.
		Then we have
		\[
			\limsup\limits_{x\to\bar x}\mathcal K(x)
			=
			\limsup\limits_{x\to\bar x}\mathcal M(x,G(x)-G(\bar x))
			\subset
			\limsup\limits_{x\to\bar x,\,y\to 0}\mathcal M(x,y)
			\subset
			\mathcal M(\bar x,0)
			=
			\mathcal K(\bar x)
		\]
		by continuity of $G$.\\
	$[\Longleftarrow]$ Assume that \eqref{eq:AM_regularity_for_disjunctive_programs} holds.
		Furthermore, choose $\{x_k\}_{k\in\N},\{x_k^*\}_{k\in\N}\subset\R^n$ and
		$\{y_k\}_{k\in\N}\subset\R^m$ such that $x_k\to\bar x$, $x_k^*\to x^*$ for some 
		$x^*\in\R^n$, $y_k\to 0$, as well as 
		$x^*_k\in G'(x_k)^\top\mathcal N_K(G(x_k)-y_k)$ for all $k\in\N$ hold.
		Then we find a sequence $\{\lambda_k\}_{k\in\N}\subset\R^m$ such that
		$x^*_k=G'(x_k)^\top\lambda_k$ and $\lambda_k\in\mathcal N_K(G(x_k)-y_k)$ are valid
		for all $k\in\N$. Exploiting that $K$ is the union of finitely many convex polyhedral
		sets, we can use similar arguments as in the proof of \cref{thm:polyhedrality_and_AMS_regularity}
		in order to find a convex, polyhedral cone $P\subset\R^m$ 
		which satisfies $P\subset\mathcal N_K(G(x_k)-y_k)$ and
		$\lambda_k\in P$ along a subsequence (without relabeling).
		The robustness of the limiting normal cone yields $P\subset\mathcal N_K(G(\bar x))$
		due to $G(x_k)-y_k\to G(\bar x)$ as 
		$k\to\infty$. Thus, we have $\{\lambda_k\}_{k\in\N}\subset\mathcal N_K(G(\bar x))$
		and, consequently, $x^*_k\in\mathcal K(x_k)$ for all $k\in\N$.
		By means of \eqref{eq:AM_regularity_for_disjunctive_programs}, we find
		$x^*\in\mathcal K(\bar x)=\mathcal M(\bar x,0)$, i.e., $\bar x$ is AM-regular.
\end{proof}

The following example points out that the assertion of \cref{thm:AM_regularity_for_disjunctive_programs}
does not need to be true whenever the set $K$ is not disjunctive, i.e., in this setting,
\eqref{eq:AM_regularity_for_disjunctive_programs} does not provide a constraint qualification
in general.
\begin{example}\label{ex:non_disjunctive_rhs}
	We investigate the setting where $G\colon\R\to\R^2$ is given by $G(x):=(x,0)$ for
	all $x\in\R$ and $K\subset\R^2$ is given by $K:=\{y\in\R^2\,|\,y_2\geq y_1^2\}$.
	Obviously, $K$ is not of disjunctive structure. The only feasible point of the
	associated constraint system $G(x)\in K$ is $\bar x:=0$. 
	The mapping $\mathcal K$ from \cref{thm:AM_regularity_for_disjunctive_programs}
	is given by $\mathcal K(x)\equiv \{0\}$ in this situation which is why the
	condition \eqref{eq:AM_regularity_for_disjunctive_programs} holds trivially.
	On the other hand, one can check
	\[
		\forall k\in\N\colon\quad
		\mathcal M\bigl(\tfrac1k,(0,-\tfrac1{k^2})\bigr)=\R_+\qquad
		\mathcal M\bigl(-\tfrac1k,(0,-\tfrac1{k^2})\bigr)=\R_-
	\] 
	and $\mathcal M(\bar x,(0,0))=\{0\}$, i.e., $\bar x$ is not AM-regular.\\
	Let us mention that one could also check the validity of the constraint
	qualification from \cref{rem:convexity} which applies to the present
	situation since $K$ is convex and $G$ is continuously differentiable.
	The latter, however, is violated as well.
\end{example}

In the literature on disjunctive programs, there exist two other weak constraint
qualifications which we will recall below, see \cite[Definition~6]{FlegelKanzowOutrata2007}.
\begin{definition}\label{def:GACQ_and_GGCQ}
	Fix a feasible point $\bar x\in M$.
	We define the \emph{linearization cone} to $M$ at $\bar x$ as stated below:
	\[
		\mathcal L_M(\bar x):=\{d\in\R^n\,|\,G'(\bar x)d\in\mathcal T_K(G(\bar x))\}.
	\]
	We say that
	\begin{enumerate}
		\item[(a)] the \emph{generalized Abadie constraint qualification} (GACQ) holds at
			the point $\bar x$ whenever the relation $\mathcal T_M(\bar x)=\mathcal L_M(\bar x)$
			is valid,
		\item[(b)] the \emph{generalized Guignard constraint qualification} (GGCQ) holds at the point
			$\bar x$ whenever the relation $\widehat{\mathcal N}_M(\bar x)=\mathcal L_M(\bar x)^\circ$
			is valid.
	\end{enumerate}
\end{definition}

Let us briefly mention that the linearization cone introduced above is, by the special structure
of $K$, also polyhedral in the sense that it is the union of finitely many convex polyhedral
cones. This is a simple consequence of
\[
	\mathcal T_K(G(\bar x))=\bigcup\limits_{i\in J(\bar x)}\mathcal T_{D_i}(G(\bar x))
\]
where we used $J(\bar x):=\{i\in\{1,\ldots,p\}\,|\,G(\bar x)\in D_i\}$ and $\bar x\in M$,
see \cite[Table~4.1]{AubinFrankowska2009}.
It has been shown in \cite[Theorem~7]{FlegelKanzowOutrata2007} that whenever $\bar x\in M$ is
a local minimizer of \eqref{eq:disjunctive_program} where GGCQ holds, then $\bar x$ is already
an M-stationary point. As pointed out in \cite{FlegelKanzowOutrata2007}, this result does not
need to hold anymore whenever continuous differentiability of $f$ is replaced by local
Lipschitz continuity.

Clearly, one could also define GACQ and GGCQ in the situation where $K$ is a general closed set.
In this case, however, GGCQ on its own does not necessarily provide a constraint qualification
ensuring M-stationarity of local minimizers.
As pointed out in \cite[Proposition~3]{BenkoGfrerer2017}, some additional metric subregularity
of a linearized feasibility mapping is needed in this more general situation, see
\cite{Gfrerer2019} as well, and the latter is inherent whenever $K$ is of disjunctive structure
due to Robinson's classical result on the inherent calmness of polyhedral set-valued mappings.

Let us now focus on \eqref{eq:disjunctive_program} again.
Let us fix one of its feasible points $\bar x\in M$.
It is well known that metric subregularity of the feasibility mapping
$\Phi\colon\R^n\tto\R^m$, given by $\Phi(x)=G(x)-K$ for all $x\in\R^n$ in the present situation,
at $(\bar x,0)$ implies validity of GACQ which, in turn, implies validity of GGCQ, see
\cite[formula~(13)]{FlegelKanzowOutrata2007}.
Noting that \eqref{eq:disjunctive_program} covers standard nonlinear problems while 
AM-regularity coincides with CCP in this setting, see \cref{ex:recovering_CCP}, the considerations
from \cite[Section~4.2]{AndreaniMartinezRamosSilva2016} show that validity of GACQ at $\bar x$ is
not sufficient for AM-regularity of $\bar x$. 
In the following example, we show that validity of AM-regularity does not need to imply validity
of GGCQ.
\begin{example}\label{eq:AM_regularity_and_GGCQ}
	Let us consider the mapping $G\colon\R\to\R^2$ given by $G(x):=(x,x^3)$ for all $x\in\R$
	as well as the disjunctive set $K:=D_1\cup D_2$ where $D_1:=\R_-\times\R$ and $D_2:=\R_+\times\R_-$
	hold. 
	In this situation, we have $M=(-\infty,0]$. Let us fix $\bar x:=0$.
	One can easily check that $\mathcal T_K(G(\bar x))=K$ holds. 
	We conclude 
	\[
		\mathcal L_M(\bar x)=\{d\in\R\,|\,(d,0)\in K\}=\R,
	\]
	and this shows that GACQ and GGCQ are violated at $\bar x$
	since we have $\mathcal T_M(\bar x)=\R_-$.
	On the other hand, we have 
	\[
		G'(x)^\top\mathcal N_K(G(\bar x))
		=
		\{\lambda_1+3x^2\lambda_2\,|\,(\lambda_1,\lambda_2)\in(\R_+\times\{0\})\cup(\{0\}\times\R_+)\}
		=
		\R_+
	\]
	for each $x\in\R$ and, thus, due to \cref{thm:AM_regularity_for_disjunctive_programs},
	$\bar x$ is AM-regular.
\end{example}

The above considerations show that AM-regularity for disjunctive programs is not related to
the constraint qualifications GACQ and GGCQ. In the particular case of MPCCs, this
already has been observed in \cite[Section~4]{Ramos2019}. Due to the results of
\cref{sec:geometric_constraints}, AM-regularity is generally weaker than NNAMCQ, i.e.,
\[
	G'(\bar x)^\top\lambda=0,\,\lambda\in \mathcal N_K(G(\bar x))
	\quad
	\Longrightarrow
	\quad
	\lambda=0,
\]
and the later is, again, weaker than the problem-tailored version of the linear independence
constraint qualification discussed in \cite{Mehlitz2019b}.

\subsection{Variational calculus and asymptotic regularity}\label{sec:variational_calculus}

In this section, we are going to show how the concept of asymptotic regularity can be used to
establish some fundamental calculus rules for limiting normals and the 
limiting coderivative.

First, we show that asymptotic regularity may serve as a sufficient condition for
the validity of the intersection rule for limiting normals.
\begin{theorem}\label{thm:intersection_rule}
	Let $K,C\subset\R^n$ be closed sets and fix $\bar x\in K\cap C$.
	Suppose that the qualification condition
	\begin{equation}\label{eq:sequential_CQ_for_intersections}
		\limsup\limits_{x\to\bar x,\,x'\to\bar x}
		\bigl(\mathcal N_{K}(x)+\mathcal N_C(x')\bigr)
		\subset
		\mathcal N_K(\bar x)+\mathcal N_C(\bar x)
	\end{equation}
	holds. Then we have
	\[
		\mathcal N_{K\cap C}(\bar x)
		\subset
		\mathcal N_K(\bar x)+\mathcal N_C(\bar x).
	\]
\end{theorem}
\begin{proof}
	This result is a simple consequence of our considerations from
	\cref{sec:geometric_constraints} and
	\cref{thm:representation_limiting_normal_cone_d_AMS_regularity} 
	when fixing $G\colon\R^n\to\R^n$ to be the identity mapping.
	Indeed, validity of \eqref{eq:sequential_CQ_for_intersections} 
	is equivalent to the validity of dAM-regularity for the
	constraint system $M:=K\cap C$.
\end{proof}

Following the structure of the books \cite{Mordukhovich2006,Mordukhovich2018},
the intersection rule provides the fundamental basis of the overall variational 
calculus. Classically, validity of the intersection rule at some point 
$\bar x\in K\cap C$ is guaranteed by
the so-called normal qualification condition
\[
	\mathcal N_K(\bar x)\cap\left(-\mathcal N_C(\bar x)\right)=\{0\},
\]
and the latter is equivalent to metric regularity of the mapping
$x\tto(x-K)\times(x-C)$ at $(\bar x,(0,0))$. 
Following the arguments from \cref{sec:geometric_constraints}, metric
subregularity of this mapping at $(\bar x,(0,0))$ is already enough to
guarantee validity of the intersection rule. Apart from these classical
results, \cref{thm:intersection_rule} shows that the intersection rule
is also valid in the presence of the \emph{asymptotic stability condition}
\eqref{eq:sequential_CQ_for_intersections} which originates from the 
notion of asymptotic regularity. Keeping in mind our results from
\cref{sec:asymptotic_regularity}, \eqref{eq:sequential_CQ_for_intersections} is
independent of the aforementioned metric subregularity condition and, thus,
provides a new approach to the variational calculus. Exemplary, we will
show how the coderivative sum and chain rule can be derived in the
presence of asymptotic stability conditions.

Let us note that validity of \eqref{eq:sequential_CQ_for_intersections} 
is equivalent to 
\[
	\limsup\limits_{x\to\bar x,\,x'\to\bar x}
	\left(\widehat{\mathcal N}_{K}(x)+\widehat{\mathcal N}_C(x')\right)
	\subset
	\mathcal N_K(\bar x)+\mathcal N_C(\bar x)
\]
by definition of the limiting normal cone.
Thus, a direct proof of the intersection rule for limiting normals
under validity of \eqref{eq:sequential_CQ_for_intersections} can
be obtained from the \emph{fuzzy} intersection rule for regular
normals, see \cite[Lemma~3.1]{Mordukhovich2006}, and a simple
diagonal sequence argument.

Next, we will inspect how validity of the coderivative sum rule can be
guaranteed under an asymptotic stability condition.
Let us mention that in \cite[Theorem~3.10]{Mordukhovich2006},
\cite[Theorem~3.9]{Mordukhovich2018}, or \cite[Theorem~10.41]{RockafellarWets1998},
the coderivative sum rule has been derived under validity of the Mordukhovich criterion.
In order to proceed, we fix set-valued mappings $S_1,S_2\colon\R^n\tto\R^m$
with closed graphs and consider their sum mapping $S\colon\R^n\tto\R^m$ given by
\[
	\forall x\in\R^n\colon\quad
	S(x):=S_1(x)+S_2(x).
\]
Furthermore, we make use of the \emph{intermediate} mapping 
$\Xi\colon\R^n\times\R^m\tto\R^m\times\R^m$ given by
\begin{equation}\label{eq:intermediate_map}
	\forall x\in\R^n\,\forall y\in\R^m\colon\quad
	\Xi(x,y):=\{
					(y_1,y_2)\in\R^m\times\R^m\,|\,
					y_1+y_2=y,\, y_1\in S_1(x),\,y_2\in S_2(x)
				\}.
\end{equation}
Observe that we have $\dom\Xi=\gph S$.
\begin{theorem}\label{thm:coderivative_sum_rule}
	Fix some point $(\bar x,\bar y)\in\gph S$.
	Then the following assertions hold.
	\begin{enumerate}
		\item[(a)] Assume that there exists $(\bar y_1,\bar y_2)\in\Xi(\bar x,\bar y)$
			such that $\Xi$ is inner semicontinuous at $((\bar x,\bar y),(\bar y_1,\bar y_2))$.
			Furthermore, let the qualification condition
			\begin{equation}\label{eq:sequential_CQ_sum}
				\begin{aligned}
				\limsup\limits_{
				\substack{
					(x_1,y_1)\to(\bar x,\bar y_1)\\
					(x_2,y_2)\to(\bar x,\bar y_2)\\
					(y_{1}^*,y_2^*)\to (\bar y_1^*,\bar y^*_2)
					}
				}
				\bigl(
					D^*S_1(x_1,y_1)(y_{1}^*)&+D^*S_2(x_2,y_2)(y_{2}^*)				
				\bigr)
				\\
				&
				\subset
					D^*S_1(\bar x,\bar y_1)(\bar y_1^*)+D^*S_2(\bar x,\bar y_2)(\bar y_2^*)				
				\end{aligned}
			\end{equation}
			hold for all $\bar y^*_1,\bar y^*_2\in\R^m$. Then, for all $y^*\in\R^m$, we have
			\[
				D^*S(\bar x,\bar y)(y^*)
				\subset
				D^*S_1(\bar x,\bar y_1)(y^*)+D^*S_2(\bar x,\bar y_2)(y^*).
			\]
		\item[(b)] Assume that $\Xi$ is inner semicompact at $(\bar x,\bar y)$.
			Furthermore, let the qualification condition \eqref{eq:sequential_CQ_sum}
			hold for each $(\bar y_1,\bar y_2)\in\Xi(\bar x,\bar y)$ 
			and all $\bar y_1^*,\bar y_2^*\in\R^m$.
			Then, for all $y^*\in\R^m$, we have
			\[
				D^*S(\bar x,\bar y)(y^*)
				\subset
				\bigcup\limits_{(\bar y_1,\bar y_2)\in\Xi(\bar x,\bar y)}
				\bigl(D^*S_1(\bar x,\bar y_1)(y^*)+D^*S_2(\bar x,\bar y_2)(y^*)\bigr).
			\]
	\end{enumerate}
\end{theorem}
\begin{proof}
	The proof essentially relies on the normal cone intersection rule from
	\cref{thm:intersection_rule} and adapts the arguments used to verify 
	\cite[Theorem~3.10]{Mordukhovich2006}.
	\begin{enumerate}
		\item[(a)] Fix $x^*\in D^*S(\bar x,\bar y)(y^*)$ for an arbitrarily chosen $y^*\in\R^m$. 
			Mimicking the proof of \cite[Theorem~3.10(i)]{Mordukhovich2006} and 
			exploiting the postulated inner semicontinuity of $\Xi$, we find
			the relation 
			$(x^*,-y^*,-y^*)\in\mathcal N_{\Omega_1\cap\Omega_2}(\bar x,\bar y_1,\bar y_2)$
			where we used the closed sets $\Omega_1,\Omega_2\subset\R^n\times\R^m\times\R^m$
			given by $\Omega_i:=\{(x,y_1,y_2)\,|\,y_i\in S_i(x)\}$, $i=1,2$.		
			
			Let us now show that \eqref{eq:sequential_CQ_sum} is sufficient for the applicability
			of the normal cone intersection rule from \cref{thm:intersection_rule} 
			for the estimation of $\mathcal N_{\Omega_1\cap\Omega_2}(\bar x,\bar y_1,\bar y_2)$
			from above.
			Therefore, we show that \eqref{eq:sequential_CQ_for_intersections} holds for the
			situation at hand. Choose sequences $\{x^*_k\}_{k\in\N}\subset\R^n$,
			$\{y^*_{k,1}\}_{k\in\N},\{y^*_{k,2}\}_{k\in\N}\subset\R^m$ as well as
			$\{x_{k,1}\}_{k\in\N},\{x_{k,2}\}_{k\in\N}\subset\R^n$ and
			$\{y_{k,1}^1\}_{k\in\N},\{y_{k,2}^1\}_{k\in\N},
			\{y_{k,1}^2\}_{k\in\N},\{y_{k,2}^2\}_{k\in\N}\subset\R^m$ such that
			$x_{k,1}\to\bar x$, $x_{k,2}\to\bar x$, $y_{k,1}^1\to\bar y_1$, $y_{k,2}^1\to\bar y_2$,
			$y_{k,1}^2\to\bar y_1$, $y_{k,2}^2\to\bar y_2$,
			$x^*_k\to x^*$ for some $x^*\in\R^n$, $y^*_{k,1}\to y^*_1$ as well as 
			$y^*_{k,2}\to y^*_2$ for some $y^*_1,y^*_2\in\R^m$, and
			\[
				(x^*_k,y^*_{k,1},y^*_{k,2})
				\in  
				\mathcal N_{\Omega_1}(x_{k,1},y_{k,1}^1,y_{k,2}^1)
				+
				\mathcal N_{\Omega_2}(x_{k,2},y_{k,1}^2,y_{k,2}^2)
			\]
			for all $k\in\N$ hold. By construction of $\Omega_1$ and $\Omega_2$, this 
			guarantees the existence of sequences 
			$\{x_{k,1}^*\}_{k\in\N},\{x_{k,2}^*\}_{k\in\N}\subset\R^n$
			such that $x^*_k=x^*_{k,1}+x^*_{k,2}$ as well as
			$(x^*_{k,i},y^*_{k,i})\in\mathcal N_{\gph S_i}(x_{k,i},y_{k,i}^i)$, $i=1,2$,
			for all $k\in\N$ hold. This leads to
			\[
				\forall k\in \N\colon\quad
				x_k^*\in D^*S_1(x_{k,1},y_{k,1}^1)(-y_{k,1}^*)+D^*S_2(x_{k,2},y_{k,2}^2)(-y_{k,2}^*).
			\]
			Due to validity of \eqref{eq:sequential_CQ_sum}, we find 
			$x^*\in D^*S_1(\bar x,\bar y_1)(-y_1^*)+D^*S_2(\bar x,\bar y_2)(-y_2^*)$,
			i.e., there are points $x_1^*,x_2^*\in\R^n$ with 
			$(x^*_i,y^*_i)\in\mathcal N_{\gph S_i}(\bar x,\bar y_i)$, $i=1,2$,
			and $x^*=x_1^*+x_2^*$.
			Particularly, we have 
			\[
				(x^*,y^*_1,y^*_2)\in
				\mathcal N_{\Omega_1}(\bar x,\bar y_1,\bar y_2)
				+
				\mathcal N_{\Omega_2}(\bar x,\bar y_1,\bar y_2).
			\]
			
			Due to the above considerations, 
			we can apply \cref{thm:intersection_rule} in order to obtain
			\[
				(x^*,-y^*,-y^*)\in
				\mathcal N_{\Omega_1}(\bar x,\bar y_1,\bar y_2)
				+
				\mathcal N_{\Omega_2}(\bar x,\bar y_1,\bar y_2).
			\]
			Now, the claim follows by definition of the sets $\Omega_1$ and $\Omega_2$,
			cf.\ \cite[proof of Theorem~3.10]{Mordukhovich2006}.
		\item[(b)] Fixing $x^*\in D^*S(\bar x,\bar y)(y^*)$ for an arbitrarily chosen
			$y^*\in\R^m$, the inner semicompactness of $\Xi$ at $(\bar x,\bar y)$ can
			be used to obtain 
			\[
				(x^*,-y^*,-y^*)
				\in
				\bigcup\limits_{(\bar y_1,\bar y_2)\in\Xi(\bar x,\bar y)}
				\mathcal N_{\Omega_1\cap\Omega_2}(\bar x,\bar y_1,\bar y_2),
			\]
			see \cite[proof of Theorem~3.10(ii)]{Mordukhovich2006} as well.
			Proceeding as in the proof of (a), the claim follows.
	\end{enumerate}
\end{proof}

Finally, we would like to take a look at the coderivative chain rule.
Therefore, let us consider set-valued mappings $T_1\colon\R^n\tto\R^m$
and $T_2\colon\R^m\tto\R^\ell$ with closed graphs as well as their composition 
$T\colon\R^n\tto\R^\ell$ given by
\[
	\forall x\in\R^n\colon\quad
	T(x):=\bigcup\limits_{y\in T_1(x)}T_2(y).
\]
Again, we will make use of an \emph{intermediate} mapping $\Theta\colon\R^n\times\R^\ell\tto\R^m$
which is given as stated below:
\[
	\forall x\in\R^n\,\forall z\in\R^\ell\colon\quad
	\Theta(x,z):=\{y\in T_1(x)\,|\,z\in T_2(y)\}.
\]
Once more, we note that $\gph T=\dom \Theta$ is valid. Similar as in
\cite[Theorem~3.13]{Mordukhovich2006} or \cite[Theorem~3.11]{Mordukhovich2018},
we will derive the coderivative chain rule from the coderivative sum rule.
Exploiting \cref{thm:coderivative_sum_rule} for that purpose, we will see
that validity of the chain rule can be guaranteed in the presence of an asymptotic stability condition. In \cite{Mordukhovich2006,Mordukhovich2018} or
\cite[Theorem~10.37]{RockafellarWets1998}, a condition related to the
Mordukhovich criterion has been imposed for that purpose.
\begin{theorem}\label{thm:coderivative_chain_rule}
	Fix some point $(\bar x,\bar z)\in\gph T$.
	Then the following assertions hold.
	\begin{enumerate}
		\item[(a)] Assume that there exists $\bar y\in\Theta(\bar x,\bar z)$ such that
			$\Theta$ is inner semicontinuous at $((\bar x,\bar z),\bar y)$.
			Furthermore, let the qualification condition
			\begin{equation}\label{eq:sequential_CQ_chain_rule}
				\begin{aligned}
					\limsup\limits_{
						\substack{
							(x,y^1)\to(\bar x,\bar y)\\
							(y^2,z)\to(\bar y,\bar z)\\
							(x^*,z^*)\to(\bar x^*,\bar z^*)
						}
					}
					\bigl(
						D^*T_2(y^2,z)(z^*)&-(D^*T_1(x,y^1))^{-1}(x^*)
					\bigr)\\
					&
					\subset
					D^*T_2(\bar y,\bar z)(\bar z^*)-(D^*T_1(\bar x,\bar y))^{-1}(\bar x^*)
				\end{aligned}
			\end{equation}
			hold for each $\bar x^*\in\R^n$ and $\bar z^*\in\R^\ell$. 
			Then, for each $z^*\in\R^\ell$, we have
			\[
				D^*T(\bar x,\bar z)(z^*)
				\subset
				\bigcup\limits_{y^*\in D^*T_2(\bar y,\bar z)(z^*)}D^*T_1(\bar x,\bar y)(y^*).
			\]
		\item[(b)] Assume that $\Theta$ is inner semicompact at $(\bar x,\bar z)$.
			Furthermore, let the qualification condition \eqref{eq:sequential_CQ_chain_rule}
			hold for each $\bar y\in\Theta(\bar x,\bar z)$, $\bar x^*\in\R^n$, and $\bar z^*\in\R^\ell$.
			Then, for each $z^*\in\R^\ell$, we have
			\[
				D^*T(\bar x,\bar z)(z^*)
				\subset
				\bigcup\limits_{\bar y\in\Theta(\bar x,\bar z)}
				\bigcup\limits_{y^*\in D^*T_2(\bar y,\bar z)(z^*)}D^*T_1(\bar x,\bar y)(y^*).
			\]
	\end{enumerate}
\end{theorem}
\begin{proof}
	We only show validity of statement (a). Assertion (b) can be obtained in analogous way.
	For the proof of (a), we exploit the idea from \cite[proof of Theorem~3.13]{Mordukhovich2006}
	and consider the mapping $S\colon\R^n\times\R^m\tto\R^\ell$ given by
	\[
		\forall x\in\R^n\,\forall y\in\R^m\colon\quad
		S(x,y):=\Delta_{\gph T_1}(x,y)+T_2(y).
	\]
	Then \cite[Theorem~1.64]{Mordukhovich2006} yields
	\begin{equation}\label{eq:upper_estimate_coderivative_T}
		D^*T(\bar x,\bar z)(z^*)
		\subset
		\{
			x^*\in\R^n\,|\,(x^*,0)\in D^*S((\bar x,\bar y),\bar z)(z^*)
		\}
	\end{equation}
	for all $z^*\in\R^{\ell}$
	since $\Theta$ is inner semicontinuous at $((\bar x,\bar z),\bar y)$.
	
	In order to estimate the coderivative of $S$ from above, we make use of
	\cref{thm:coderivative_sum_rule}. Therefore, we introduce $S_1,S_2\colon\R^n\times\R^m\tto\R^\ell$
	by means of
	\[
		\forall x\in\R^n\,\forall y\in\R^m\colon\quad
		S_1(x,y):=\Delta_{\gph T_1}(x,y)
		\qquad
		S_2(x,y):=T_2(y).
	\]
	Next, we show that \eqref{eq:sequential_CQ_sum} holds in the present setting.
	Therefore, we make use of the formulas
	\begin{align*}
		D^*S_1((x,y),z)(z^*)
		&=
		\begin{cases}
			\mathcal N_{\gph T_1}(x,y)&z=0\\
			\varnothing					&z\neq 0
		\end{cases}
		\\
		D^*S_2((x,y),z)(z^*)
		&=
		\{0\}\times D^*T_2(y,z)(z^*)
	\end{align*}
	which, by elementary calculations, 
	hold for all $x\in\R^n$, $y\in\R^m$, and $z,z^*\in\R^\ell$.
	In order to infer validity of \eqref{eq:sequential_CQ_sum}, 
	we thus need to verify validity of
	\begin{equation}\label{eq:sequential_CQ_chain_rule_intemediate_form}
	\begin{aligned}
		\limsup\limits_
			{	\substack{
				(x,y^1)\to(\bar x,\bar y)\\
				(y^2,z)\to(\bar y,\bar z)\\
				z^*\to\bar z^*
				}
			}
		\bigl(\mathcal N_{\gph T_1}(x,y^1)&+\{0\}\times D^*T_2(y^2,z)(z^*)\bigr)\\
		& \subset
		\mathcal N_{\gph T_1}(\bar x,\bar y)+\{0\}\times D^*T_2(\bar y,\bar z)(\bar z^*)
	\end{aligned}
	\end{equation}
	for all $\bar z^*\in\R^\ell$.
	Thus, for some point $\bar z^*\in\R^{\ell}$, we fix sequences $\{x_k\}_{k\in\N}\subset\R^n$, 
	$\{y_k^1\}_{k\in\N},\{y_k^2\}_{k\in\N}\subset\R^m$, 
	$\{z_k\}_{k\in\N},\{z_k^*\}_{k\in\N}\subset\R^\ell$,
	as well as $\{x^*_k\}_{k\in\N}\subset\R^n$ and $\{y^*_k\}_{k\in\N}\subset\R^m$
	such that
	$x_k\to\bar x$, $y_k^1\to\bar y$, $y_k^2\to\bar y$, $z_k\to\bar z$, $z_k^*\to\bar z^*$,
	$x_k^*\to x^*$ and $y_k^*\to y^*$ for some $x^*\in\R^n$ and $y^*\in\R^m$, as well as
	\[
		(x^*_k,y^*_k)
		\in 
		\mathcal N_{\gph T_1}(x_k,y^1_k)+\{0\}\times D^*T_2(y^2_k,z_k)(z^*_k)
	\]
	for all $k\in\N$ hold. Keeping the definitions of the coderivative and the inverse mapping
	in mind, we find
	\[
		\forall k\in\N\colon\quad
		y^*_k\in D^*T_2(y_k^2,z_k)(z_k^*)-(D^*T_1(x_k,y_k^1))^{-1}(x_k^*).
	\]
	Inspecting \eqref{eq:sequential_CQ_chain_rule}, we obtain
	\[
		y^*\in D^*T_2(\bar y,\bar z)(\bar z^*)-(D^*T_1(\bar x,\bar y))^{-1}(x^*),
	\]
	i.e., $(x^*,y^*)\in\mathcal N_{\gph T_1}(\bar x,\bar y)+\{0\}\times D^*T_1(\bar x,\bar y)(\bar z^*)$.
	This shows validity of \eqref{eq:sequential_CQ_chain_rule_intemediate_form}.
	Observing that the intermediate mapping $\Xi$ from \eqref{eq:intermediate_map}
	is given by
	\begin{align*}
		\forall x\in\R^n\,\forall y\in\R^m\,\forall z\in\R^\ell\colon\quad
		\Xi(x,y,z)&=\{(0,z)\,|\,y\in T_1(x),\,z\in T_2(y)\}\\
				  &=\begin{cases}
				  	\{(0,z)\}	&((x,z),y)\in\gph \Theta\\ \varnothing&\text{otherwise}
				  	\end{cases}
	\end{align*}
	in this situation and, thus, 
	is trivially inner semicontinuous at $((\bar x,\bar y,\bar z),(0,\bar z))$
	by inner semicontinuity of $\Theta$ at $((\bar x,\bar z),\bar y)$,
	we can apply assertion (a) of \cref{thm:coderivative_sum_rule} in order to find
	\[
		D^*S((\bar x,\bar y),\bar z)(z^*)
		\subset
		\mathcal N_{\gph T_1}(\bar x,\bar y)
		+
		\{0\}\times D^*T_2(\bar y,\bar z)(z^*).
	\]
	Due to \eqref{eq:upper_estimate_coderivative_T}, the desired estimate
	is obtained.
\end{proof}

\section{Conclusions}\label{sec:conclusions}

In this paper, we introduced a new sequential constraint qualification, namely
AM-regularity, for nonsmooth
optimization problems. This concept has been shown to be generally weaker than metric regularity of the
associated feasibility mapping while it is not related to metric subregularity of the latter.
AM-regularity turned out to be a condition which is sufficient for the validity of the
pre-image rule from the limiting variational calculus.

We clarified how abstract constraints can be incorporated into the framework of AM-regularity
and presented some associated consequences for optimization problems with geometric constraints.
Our findings were applied to mathematical programs with disjunctive constraints as well.
This revealed that AM-regularity is a generalization of the cone-continuity property
(also referred to as AKKT-regularity)
for standard nonlinear problems and mathematical programs with complementarity constraints,
see \cite{AndreaniMartinezRamosSilva2016,Ramos2019}.
Keeping e.g.\ \cite{AndreaniFazzioSchuverdtSecchin2019,AndreaniHaeserViana2020,
BoergensKanzowMehlitzWachsmuth2019,Ramos2019} in mind, constraint qualifications of 
AM-regularity-type can be used to ensure convergence of different types of solution algorithms 
like augmented Lagrangian or relaxation methods to stationary points of several classes of
optimization problems. It is a promising subject of future research to investigate more
general algorithmic consequences of AM-regularity.

We finalized the paper by showing that asymptotic regularity provides a new approach to the limiting
variational calculus. It remains to be seen whether the resulting new asymptotic stability
conditions which ensure validity of the normal cone intersection rule, the coderivative
sum rule, or the coderivative chain rule can be used profitably in the context of
variational analysis. Following ideas from \cite{BenkoGfrererOutrata2019,Gfrerer2013}, it might
be possible to introduce a reasonable concept of \emph{directional} AM-regularity.
Such a concept may provide qualification conditions for optimization problems of type
\eqref{eq:basic_problem} and the \emph{directional} limiting variation calculus which are
even weaker than the criteria inferred from AM-regularity.



\end{document}